\documentclass{amsart}
\usepackage{graphicx}
\usepackage{amssymb,amscd,amsthm,amsxtra}
\usepackage{latexsym}
\usepackage{epsfig}

\newtheorem{thm}{Theorem}[section]

\newtheorem{lem}[thm]{Lemma}
\newtheorem{prop}[thm]{Proposition}
\theoremstyle{definition}
\newtheorem{defn}[thm]{Definition}
\theoremstyle{remark}
\newtheorem{rem}[thm]{Remark}
\numberwithin{equation}{section}

\newcommand{\R}{\mathbb R}

\newcommand{\p}{\partial}

\newcommand{\comment}[1]{}

\begin{document}

\title[Thin free boundaries]{$C^{\infty}$ regularity of certain thin free boundaries}
\author{D. De Silva}
\address{Department of Mathematics, Barnard College, Columbia University, New York, NY 10027}
\email{\tt  desilva@math.columbia.edu}
\author{O. Savin}
\address{Department of Mathematics, Columbia University, New York, NY 10027}\email{\tt  savin@math.columbia.edu}

\thanks{ D.~D.~ and O.~ S.~  are supported by the ERC starting grant project 2011 EPSILON (Elliptic PDEs and Symmetry of Interfaces and Layers for Odd Nonlinearities). D.~D. is supported by NSF grant DMS-1301535. O.~S.~ is supported by NSF grant DMS-1200701.}

\begin{abstract}
We continue our study of the free boundary regularity in the thin one-phase problem and show that $C^{2,\alpha}$ free boundaries are smooth.
\end{abstract}
\maketitle

\section{Introduction}

In this paper we investigate $C^\infty$ regularity of the free boundary in the  {\it thin one-phase} problem. In general a thin free boundary refers to a problem in which the free boundary is expected to have codimension 2.

We consider the thin one-phase problem (or thin Bernoulli problem) which consists in finding a non-negative function
$$u: \overline {B_1} \subset \R^{n+1} \to \R, \quad \quad u\in C( \overline {B_1})$$
with prescribed values $u=\varphi \ge 0$ on $\p B_1$, such that $u$ satisfies
$$\{u=0\} \subset \{x_{n+1}=0\},$$
and
\begin{equation}\label{fb}
\begin{cases}
\Delta u =0 \quad \quad \quad \mbox{in $\{u>0\}$}\\
\frac{\p u}{\p \sqrt t}=1 \quad \quad \quad \mbox{on $\Gamma:=\p_{\R^n} \{u=0\} \subset \{x_{n+1}=0\}.$}
\end{cases}
\end{equation}
We used the notation
$$\frac{\p u}{\p \sqrt t}(Z):=\lim_{t\to 0^+} \frac{u(z+ t \nu,0)}{t^{1/2}}, \quad Z=(z,0) \in \Gamma,$$
and $\nu$ denotes the outward normal to the free boundary $\Gamma$ in $\R^n$. There is an energy functional associated to this problem,
\begin{equation}\label{e}
E(u):= \int  |\nabla u|^2 \,dX + \, \frac \pi 2 \, \, \mathcal H^{n} \left(\{u>0\} \cap \{x_{n+1}=0\}\right),
\end{equation}
and solutions to \eqref{fb} are critical points for $E$.

To fix ideas we explain the situation in the simplest case $n=1$. Typically $u$ vanishes continuously on a number of segments on $\{x_2=0\}$ and $u$ is positive harmonic on the two dimensional disk away from these segments. In this case the free
boundary $\Gamma$ consists of the endpoints of these horizontal segments. A harmonic function grows on the $x_1$-axis as $a \, d^{1/2} +o(d^{1/2})$ away from its vanishing segments, for some constant $a$, where $d$ represents the distance to the zero set. The free boundary condition above requires that the constant $a$ must be $1$ for all endpoints. It can be understood as a Neumann type condition which determines the set $\{u=0\}$.

The thin one-phase free boundary problem was first considered by Caffarelli, Roquejoffre and Sire \cite{CRS} as a model of a one-phase Bernoulli type free boundary problem in the context of the fractional Laplacian. It appears in flame propagation when turbulence or long-range interactions are present. When $n=2$, the problem \eqref{fb} is related to models involving traveling wave solutions for planar cracks. In this setting $\{u=0\}$ represents the location of the crack in a 3D material and the free boundary $\Gamma$ is one-dimensional and represents the edge of the crack. For further information on this model see \cite{CRS} and the references therein.

The study of the regularity of thin one-phase free boundaries was initiated in \cite{DR}, where it was shown that ``flat" free boundaries are $C^{1,\alpha}$.  In \cite{DS1}, \cite{DS2} we  continued investigating this regularity issue. These results parallel the regularity theory for the free boundary in the classical one-phase problem and in the theory of minimal surfaces. We showed that Lipschitz free boundaries are of class $C^{2,\alpha}$ and local minimizers of $E$ have $C^{2,\alpha}$ free boundary except possibly for a small singular set of Hausdorff dimension $n-3$. In the current paper we address the issue of higher regularity of the free boundary. We prove that $C^{2,\alpha}$ free boundaries are in fact smooth.

 \begin{thm}\label{MaI}
 Assume $u$ satisfies \eqref{fb} and $\Gamma \in C^{2,\alpha}$. Then $\Gamma \in C^\infty$.
 \end{thm}

 The techniques developed in this paper are quite general and can be used to investigate the higher regularity of other thin free boundaries.  One example of thin free boundary arises in the so-called {\it thin obstacle problem} also known as the {\it Signorini problem} (see for example \cite{ACS, CSS, GP}). 
 
 The main difficulty in the thin one-phase problem occurs near the free boundary where all derivatives of $u$ blow up and the problem becomes degenerate. We discuss briefly the free boundary regularity in the case of the classical Bernoulli problem (\cite{AC, C1,C2}):
\begin{equation}\label{Bern}
\begin{cases}
\Delta u =0 \quad \mbox{in $\{u>0\},$}\\
|\nabla u|=1 \quad \mbox{on $\Gamma:=\p \{u>0\}.$}
\end{cases}
\end{equation}
 The analyticity of $C^{1,\alpha}$ free boundaries $\Gamma$ was obtained by Kinderlehrer, Nirenberg and Spruck in \cite{KNS}.
 They used the hodograph transform to reduce the problem to a nonlinear Neumann problem with fixed boundary.
 We sketch below an equivalent argument to prove higher regularity of $\Gamma$ in \eqref{Bern}. It avoids the hodograph transformation and it makes use of Schauder estimates for both a Dirichlet and a Neumann problem. We will follow this strategy also in the proof of our main result Theorem \ref{MaI}.

 Assume that $\Gamma \in C^{k+2,\alpha}$ for some $k \ge 0$. Then by Schauder estimates for the Dirichlet problem in the set $\{u>0\}$ we find
 \begin{equation}\label{bs}
 \Gamma \in C^{k+2,\alpha} \quad \Rightarrow \quad u \in C^{k+2,\alpha}.
 \end{equation}
Also, 
\begin{equation}\label{bs1}
\Delta (u_n w)=0 \quad \text{in $\{ u>0\}$} \quad \mbox{and} \quad w_\nu=0 \quad \text{on $\Gamma$}, \quad \mbox{with} \quad w:=\frac{u_i}{u_n},
\end{equation}
where the Neumann condition follows by differentiating the free boundary condition in \eqref{Bern} along $\Gamma$.
Geometrically the quotient $w$ represents the $i$-derivative of the level set of $u$ viewed as a graph in the $e_n$ direction. Since $u_n$ is harmonic we can write the equation above as an equation with coefficients in $C^{k,\alpha}$ (see \eqref{bs})
$$\Delta w + 2 \frac{\nabla u_n}{u_n} \cdot \nabla w =0.$$
Now we apply the Schauder estimates for the Neumann problem and obtain that solutions to \eqref{bs1} satisfy
$w \in C^{k+2,\alpha}$ and this gives $\Gamma \in C^{k+3,\alpha}$.

It turns out that in the thin one-phase problem the quotient $w$ still satisfies \eqref{bs1}. We prove Theorem \ref{MaI} by obtaining regularity results as \eqref{bs}-\eqref{bs1} in the context of the thin free boundary problem.

To this aim, we consider Schauder estimates at the boundary for harmonic functions in {\it slit domains}, see Theorem \ref{Schauder} for a precise statement. A slit domain is a domain in $\R^{n+1}$ from which we remove an $n$-dimensional set $\mathcal P \subset \{x_{n+1}=0\}$ (slit), with $C^{k+2,\alpha}$ boundary in $\R^n$, $\Gamma:=\p_{\R^n} \mathcal P$, $k \ge 0$.

  In the simplest case when $n=1$ and $\mathcal P$ is the negative $x_1$-axis, then a harmonic function $u$ in $B_1 \setminus \mathcal P$, even with respect to the $x_1$-axis and which vanishes continuously on $\mathcal P$, can be written near the origin as a series of homogenous harmonic functions
  $$ r^q \cos\left(q \theta \right), \quad \quad q=\frac 12,\frac 32,\frac 52,\ldots,$$
 where $r$ and $\theta$ denote the polar coordinates. In particular it follows that $u$ has an expansion at the origin of the type
$$u=U_0 \left(P(x_1,r) + O(r^{k+1+\alpha})\right), \quad \quad \quad U_0:=r^\frac 12 \cos (\theta/ 2),$$
for some polynomial $P$ of degree $k+1$, where $U_0$ denotes the first homogenous harmonic function.

In Theorem \ref{exp} we show that this expansion remains valid also for slit domains in $\R^{n+1}$ with boundary $\Gamma \in C^{k+2,\alpha}$, with $P$ a polynomial of degree $k+1$ in $x_1, \ldots, x_n$ and $r$. In this case $(r,\theta)$ denote the polar coordinates with respect to $\Gamma$.

In our next step,  we use this expansion for $u$ and obtain Schauder estimates at the boundary for solutions $w$ to the Neumann problem \eqref{bs1}. This equation is quite degenerate in our case since the weight $u_n$ is singular on $\Gamma$. We show that $w$ can be approximated in a $C^{k+2,\alpha}$ fashion by a polynomial in the variables $x_1, \ldots, x_n$ and $r$, see Theorem \ref{thin} for a precise statement.

The proofs of the Schauder estimates above use perturbation arguments, see \cite{CC}. Roughly speaking, we approximate our solutions in dyadic balls by solutions to ``constant coefficient" equations which in turn are approximated by appropriate polynomials in $x$ and $r$.

It is worth remarking that the equations we consider do not behave well under general smooth changes of coordinates.
From the expansion of $u$ near $\Gamma$ we see that one should consider changes of variables which leave $r$ and $\theta$ invariant, at least infinitesimally on $\Gamma$. For example, if we flatten the boundary $\Gamma$ and move isometrically the 2D planes perpendicular to $\Gamma$, then this change of variables has a loss of one derivative with respect to the smoothness of $\Gamma$. In the proof of Theorem \ref{thin} we use Whitney's extension theorem to overcome this technical difficulty.

The paper is organized as follows. In Sections 2,3 and 4 we introduce notation and state our main theorems from which Theorem \ref{MaI} follows. Section 5 and 6 are devoted to the proof of our Schauder estimate for solutions to Laplace's equation in slit domains. Section 7 provides the proof of the Schauder estimate for solutions to the Neumann problem. Some technical facts are proved in the Appendix.

\section{Notation and definitions}

\subsection{Notation}
We introduce some notation that we use throughout the paper.

Let $\Gamma$ be a $C^{k+2,\alpha}$ surface in $\R^n$, $k \geq 0.$ Assume for simplicity that $\Gamma$ is given by the graph of a function $g$ of $n-1$ variables
\begin{equation}\label{G}
\Gamma:=\{(x', g(x'))\}, \quad \quad \quad g:B_1' \subset \R^{n-1}\to \R,
\end{equation}
satisfying
\begin{equation*}\label{eqg}
g(0)=0, \quad \nabla_{x'} g(0) =0, \quad \|g\|_{C^{k+2,\alpha}(B_1')} \leq 1.
\end{equation*}

Let $\mathcal P$ denote the $n$ dimensional slit in $\R^{n+1}$ given by
$$\mathcal P := \{X=(x,x_{n+1}) \in B_1 \ | \ x_{n+1}=0, x_n \leq g(x')\}.$$
Notice that in the $n$ dimensional ball $B_1' \times\{0\}$ we have $\p_{\R^n} \mathcal P =\Gamma$.

Given a point $X=(x,x_{n+1})$ we denote by $d$ the signed distance in $\R^n$ from $x$ to $\Gamma$ with $d>0$ above $\Gamma$ (in the $e_n$ direction). Denote by $$r:=\sqrt{x_{n+1}^2+d^2}$$ the distance in $\R^{n+1}$ from $X$ to $\Gamma.$
We have
\begin{equation}\label{dr}
\nabla_x r= \frac d r \, \nu, \quad \quad \nu=\nabla_x d,
\end{equation}
and $\nu(x)$ represents the unit normal in $\R^n$ to the parallel surface to $\Gamma$ passing through $x$.

Let $\theta \in (-\pi, \pi]$ be the angle between the segment of length $r$ from $X$ to $\Gamma$ and  the $x$-hyperplane and define $$U_0(X):= r^{1/2}\cos \frac{\theta}{2}= \frac{1}{\sqrt 2} \sqrt{d+r}.$$
It is easy to check that
\begin{equation}\label{dr1}
\nabla_x U_0 = \frac{U_0}{2r} \,  \nu.
\end{equation}

We denote by $c$, $C$ various positive constants that depend only on $n$, $k$ and $\alpha$.

\subsection{The class $C^{k,\alpha}_{xr}$}
 In this paper we work with functions which near $\Gamma$ can be expanded as power series in the variables $x_1$, $x_2$, ..., $x_n$ and $r$. Since we deal with higher regularity we remark that these functions are not sufficiently regular when viewed in the original variable $X=(x,x_{n+1})$.
Thus we need to introduce the notion of a $C^{k,\alpha}$ function in the $(x,r)$-variables. We are interested only in power expansions at points on $\Gamma$ and for this reason we define the notion of pointwise $C^{k,\alpha}$ function in the $(x,r)$ variables.

We denote by
$$P(x,r) = \, a_{\mu m}\,  x^\mu r^m, \quad \deg P=k,$$
a polynomial of degree $k$ in the $(x,r)$ variables, and we use throughout the paper the summation convention over repeatedly indices.
Above we used the following notation:
$$x^{\mu}=x_1^{\mu_1}\ldots x_{n}^{\mu_n}, \quad |\mu| = \mu_1+ \ldots +\mu_n, \quad \quad \mu_i \ge 0.$$
Sometimes it is useful to think that $a_{\mu m}$ are defined for all indices $(\mu,m)$, by extending them to be $0$.

We also denote $$\|P\|:= \max|a_{\mu m}|.$$
\

\begin{defn}We say that a function $f: B_1 \subset \R^{n+1} \rightarrow \R$ is { \it pointwise $C^{k,\alpha}$ in the $(x,r)$-variables at $0 \in \Gamma$} and write $f \in C_{xr}^{k,\alpha}(0)$ if there exists a (tangent) polynomial $P_0(x,r)$ of degree $k$
 such that $$f(X) = P_0(x, r) + O(|X|^{k+\alpha}).$$
\end{defn}

We define $\|f\|_{C^{k,\alpha}_{xr}(0)}$ as the smallest constant $M$ such that $$\|P_0\| \leq M, \quad \mbox{and} \quad |f(X) - P_0(x,r)| \leq M|X|^{k+\alpha},$$ for all $X$ in the domain of definition.

Similarly, we may write the definition for $f$ to be pointwise $C^{k,\alpha}_{xr}$ at some other point $Z \in \Gamma$. Next we define the notion of $C^{k,\alpha}_{xr}$ on a whole subset $K \subset \Gamma$.

\

\begin{defn}Let $K \subset \Gamma.$ We say that $f \in C_{xr}^{k,\alpha} (K)$ if there exists a constant $M$ such that $f \in C_{xr}^{k,\alpha}(Z)$ for all $Z \in K$ and $\|f\|_{C_{xr}^{k,\alpha}(Z)} \leq M$ for all $Z \in K.$

The smallest $M$ in the definition above is denoted by $\|f\|_{C_{x,r}^{k,\alpha}(K)}.$
\end{defn}

\section{Harmonic functions in slit domains}
As first step towards the proof of our main Theorem \ref{MaI}, we are interested in the regularity of solutions to the Laplace equation in slit domains and their precise behavior on the edge of the slit. We collect here our main statements. First we remark that we may restrict ourself to the case when solutions are even with respect to $x_{n+1}$. Indeed, let
$$\Delta u=F \quad \mbox{in} \quad B_1 \setminus \mathcal P,$$
and $u$ vanish continuously on $\mathcal P$. We decompose $u=u_{ev}+u_{od}$ with $u_{ev}$, $u_{od}$ even respectively odd with respect to $x_{n+1}$. Notice that $u_{ev}$ and $u_{od}$ solve the Laplace equation with right hand side $F_{ev}$, respectively $F_{od}$. Since $u_{od}$ vanishes continuously on $x_{n+1}=0$, its regularity follows from the boundary regularity for Laplace equation in smooth domains. For example if $F_{od}$ is smooth, then $u_{od}$ can be expanded as a power series in $x$, $x_{n+1}$ at the origin.

Next we state our Schauder estimates in slit domains with $C^{k+2,\alpha}$ boundary.

Let $u \in C(B_1)$ be even in the $x_{n+1}$ coordinate, with $\|u\|_{L^\infty} \leq 1$, and
\begin{equation}\label{FB}
\begin{cases}
\Delta u = \dfrac{U_0}{r} f \quad \quad \text{in $B_1 \setminus \mathcal P$}\\
u=0 \quad \text{on $\mathcal P$.}
\end{cases}
\end{equation}

\begin{thm}[Schauder estimates in slit domains]\label{Schauder}
Let $\Gamma$, $u$ satisfy \eqref{G}, \eqref{FB} with $$f \in C^{k,\alpha}_{xr} (\Gamma \cap B_1), \quad \quad \|f\|_{C^{k,\alpha}_{xr} (\Gamma \cap B_1)} \leq 1.$$ Then,
\begin{equation}\label{Su}\left \|\dfrac{u}{U_0} \right \|_{C^{k+1,\alpha}_{xr} (\Gamma \cap B_{1/2})} \leq C \end{equation} and
\begin{equation}\label{SDu}\left \|\dfrac{\nabla_x u}{(U_0/r)} \right \|_{C^{k+1,\alpha}_{xr} ( \Gamma \cap B_{1/2})} \leq C \end{equation}
with $C$ a constant depending only on $n$, $k$ and $\alpha$.
\end{thm}

 The Theorem above states that $u$ satisfies the following expansion at $0 \in \Gamma$
$$ u(X)= U_0(X) \;(P_0(x,r) + O(|X|^{k+1+\alpha})),$$ for some polynomial $P_0(x,r)$ of degree $k+1$. The derivatives $u_i$ are in fact obtained by differentiating formally this expansion in the $x_i$ direction. Using \eqref{dr}-\eqref{dr1} we have
\begin{equation}\label{form}
\nabla_x u = \frac{U_0}{r} \left [ \frac 12 P_0 \, \nu + r  \, \p_x P_0 + (\p_r P_0) \, d \, \nu + O(|X|^{k+1+\alpha}) \right].
\end{equation}
 Since $\nu, d \in C_x^{k+1,\alpha}$ we obtain
$$u_i= \dfrac{U_0}{r}(P_0^i(x,r) + O(|X|^{k+1+\alpha})), \quad \quad  \deg P_0^i=k+1,$$
for some polynomial $P_0^i$.

The boundary Harnack estimate for harmonic functions in slit domains with Lipschitz boundary (in $\R^n$) states that the quotient of two positive solutions which vanish continuously on $\mathcal P$ is H\"older continuous (see \cite{CFMS}). Theorem \ref{Schauder} can be understood as an optimal boundary Harnack estimate in the case when the boundary of the slit $\Gamma$ has higher $C^{k,\alpha}$ regularity.

We prove Theorem \ref{Schauder} in Section \ref{s5} for the case $k=0$ and in Section \ref{s6} for general $k$. We mention that the theorem holds also for $k=-1$, i.e. when $\Gamma \in C^{1,\alpha}$, see \cite{DS3}.

\begin{rem} In Section 5,
we will show also that if $f$ is more regular away from $\Gamma$, say such that it guarantees the existence of second derivatives of $u$ locally, then also $u_{ij}$ are obtained by formally differentiating the expression above. In particular if $f$ is a $C^\alpha$ function in the $X$ variable in the whole $B_1,$ then
$$ u_{ij}= \dfrac{U_0}{r^3}(P_0^{ij}(x,r) + O(|X|^{k+2+\alpha})),$$ with $P_0^{ij}$ a sum of monomials with degrees between $1$ and $k+2.$

In the case of harmonic functions  ($f \equiv 0$) we can obtain all derivatives of order $|\mu| \leq k+2$ by differentiating formally,
$$ D_x^\mu u = \dfrac{U_0}{r^{2|\mu| -1}}(P_0^{\mu}(x,r) + O(|X|^{k+|\mu|+\alpha})),$$ with $P_0^\mu$ having monomials with degrees between $|\mu| -1$ and $k+|\mu|.$

\end{rem}

We also state the polynomial expansion near $\Gamma$ for general harmonic functions (not necessarily even) in slit domains with $C^{k+2,\alpha}$ boundary, since it is of interest on its own. The expansion involves the first two harmonic functions in 2D i.e. $$U_0=r^\frac 12 \cos ( \theta / 2 ) \quad \mbox{even}, \quad \quad x_{n+1}=r \sin \theta \quad \quad \mbox{odd}$$
 multiplied by powers of $r$. Precisely we have
\begin{thm}[Expansion of harmonic functions]\label{exp}
Assume $\Gamma \in C^{k+2,\alpha}$, $k \ge 0$, and $u \in C(B_1)$ satisfies
$$\Delta u =0 \quad \mbox{in} \quad B_1 \setminus \mathcal P, \quad \quad u=0 \quad \mbox{on} \quad \mathcal P.$$ There exist functions
$$a_j(x) \in C^{k+1-j,\alpha}\quad \mbox{and} \quad \quad b_j(x) \in C^{\infty},$$ such that for all $X \in B_{1/2}$
$$\left |u(X)-  U_0 \left( \sum_{j=0}^{k+1} a_j(x)\, \, r^j \right )-x_{n+1} \left( \sum_{j\le k/2} b_j(x)\, \, x_{n+1}^{2j} \right) \right | \le M \, U_0 \, r^{k+1+\alpha} .$$
 The constant $M$ and the norms of $a_j$, $b_j$ depend on $\|u\|_{L^\infty}$, $\|\Gamma\|_{C^{k+2,\alpha}}$, $k$, $\alpha$ and $n$.
\end{thm}
The first and second term above approximate the even respectively odd part of $u$. More generally we will show that if
$$\Delta u=F \quad \mbox{in} \quad B_1 \setminus \mathcal P, \quad \quad F \in C^{k,\alpha},$$
then
$$u=U_0 \left(\sum_{m=0}^{k+1}a_m(x)r^m +O(r^{k+1+\alpha})\right) + x_{n+1}\left (\sum_{m =0}^k b_m(x) x_{n+1}^m \right ) ,$$
for functions $a_m \in C^{k+1-m,\alpha}$, $b_m \in C^{k-m, \alpha +\frac12}$.

\section{The thin one-phase problem}

In this Section we show that our main Theorem \ref{MaI} follows from Theorem \ref{Schauder} and a Schauder estimate for a Neumann-type problem which we also state here.

Assume $u \in C(B_1)$ is a solution to the thin one-phase free boundary problem
\begin{equation}\label{TOF}
\begin{cases}
\Delta u = 0 \quad \quad \text{in $B_1 \setminus \mathcal P$}, \quad \quad u=0 \quad \mbox{on $\mathcal P$,}\\
\frac{\partial u}{\partial U_0}=1 \quad \text{on $\Gamma$,}
\end{cases}
\end{equation}
where
$$\frac{\partial u}{\partial U_0}(Z):=\lim_{t \to 0^+} \frac{u(z+t \nu,0)}{t^{1/2}}, \quad \quad Z=(z,0)\in \Gamma.$$

We assume that $\Gamma \in C^{2,\alpha}$ satisfies \eqref{G}, and after replacing $u$ by its even part, we also assume that $u$ is even in $x_{n+1}$. By Theorem \ref{Schauder}, at a point $Z \in \Gamma$ we have the expansion
$$u(X)=U_0(X) \left(P_Z(x,r) +O(|X-Z|^{1+\alpha}) \right), \quad \quad \deg P=1.$$
Notice that
$$P_Z(z,0)=\frac{\partial u}{\partial U_0}(Z).$$

\subsection{Equation for the quotient $w$}
We show that the quotient $$w:=\frac{u_i}{u_n}, \quad \quad w\in C(B_1)$$
satisfies the following problem with Neumann boundary condition on $\Gamma$:
\begin{equation}\label{NeP}
\begin{cases}
\Delta (u_n w)=0 \quad \quad \mbox{in} \quad B_1 \setminus \mathcal P, \\

w_\nu=0 \quad \quad \quad \mbox{on $\Gamma$,}
\end{cases}
\end{equation}
with
$$w_\nu(Z):=\lim_{t \to 0^+} \frac{w(z+t \nu,0)-w(z,0)}{t}, \quad \quad Z \in \Gamma.$$
Notice that $w$ represents the derivative in the $-e_i$ direction of the level sets of $u$ viewed as graphs in the $e_n$ direction. In particular on $\Gamma$
\begin{equation}\label{wg}
w(Z)=-g_i(z'),
\end{equation}
and this gives the relation between the regularity of $w$ on $\Gamma$ and the regularity of $\Gamma$ itself.

First, we remark that $w$ is indeed continuous in $B_1$. In fact from \eqref{form} it follows that $u_n>0$ is a neighborhood of zero. After a dilation we can assume that this is true in $B_1.$ Now, 
again from \eqref{form} we conclude that $w$ is continuous on $\Gamma$, and boundary Harnack inequality gives the continuity of $w$ on the slit $\mathcal P$.

Next we check the Neumann condition for $w$. Let
$$P_0(x,r)=a_0+ a_ix_i + a_{n+1}r$$ be the polynomial in the expansion of $u$ at $0$. From the free boundary condition we find, $$1=P_0(z,0) +O(|z|^{1+\alpha}), \quad \quad Z \in \Gamma $$ thus, using that $\nabla_{x'} g(0)=0,$ we get
$$a_0=1, \quad a_i=0 \quad 1 \le i \le n-1.$$
By  \eqref{form} we see that on the line $t e_n$, $\nu=e_n$ hence
$$u_i (t e_n) = t^{-\frac 12} O(t^{1+\alpha}) , \quad \quad 1 \le i \le n-1,$$
which gives
$$(t^ \frac 12 u_i)(0)=\frac{d}{dt}(t^\frac 12 u_i)(0)=0.$$
For any vector $\tau=(\tau_1,..,\tau_n) \in \R^n$, $\tau_n \ne 0$ we obtain that 
$$\frac{d}{dt}\left[\log (t^ \frac 12u_\tau)\right](0) \quad \quad \mbox{does not depend on $\tau$},$$
and it follows that for any two vectors in $\R^n$, $\tau$, $\sigma$, with $\tau_n \ne 0$ we have
$$ \left(\frac {u_\sigma}{u_\tau}\right)_{e_n}(0)=0.$$

Thus $w$ solves the Neumann problem \eqref{NeP}.
We will prove the following estimate for solutions to such Neumann problem.

\begin{thm}\label{thin}Let $\Gamma \in C^{k+2,\alpha}$ satisfy \eqref{G} and let $u$ be a harmonic function in $B_1 \setminus \mathcal P$, even in $x_{n+1}$, such that $\frac 12 U_0 \le u \le 2 U_0$. Assume $w \in C(B_1)$, even in $x_{n+1}$, solves the Neumann problem
$$\begin{cases}
\Delta(u_n w) = 0 \quad \text{in $B_1 \setminus \mathcal P,$}\\
w_{\nu} =0 \quad \quad \quad \text{on $\Gamma.$}
\end{cases}$$
Then $w \in C_{x,r}^{k+2, \alpha}(\Gamma)$ and
$$\|w\|_{C_{x,r}^{k+2, \alpha}(\Gamma\cap B_{1/2})} \le C \|w\|_{L^\infty(B_1)},$$
with $C$ depending only on $n,k, \alpha$.
\end{thm}
Clearly the function $w$ is a $C^{k+2,\alpha}$ function when restricted to $\Gamma$.

\

Theorem \ref{MaI} is a direct corollary of Theorem \ref{Schauder} and Theorem \ref{thin}. Indeed, if $u$ is a solution to the thin one-phase problem and $\Gamma \in C^{k+2,\alpha}$, then by Theorem \ref{Schauder} and the free boundary condition, $u$ satisfies the assumptions of Theorem \ref{thin} (after a dilation.) We then apply Theorem \ref{thin} to the quotient $w=u_i/u_n$ and obtain (see \eqref{wg}) that in fact $\Gamma \in C^{k+3,\alpha}$.

\begin{rem}\label{rf}
From the proof of Theorem \ref{thin} it follows that the conclusion holds if the homogenous Neumann condition is replaced by $w_\nu \in C^{k+1,\alpha}$ on $\Gamma$.
\end{rem}

\begin{rem}
Theorems \ref{Schauder} and \ref{thin} apply also in the case $k=-1$, that is when $\Gamma \in C^{1,\alpha}$ (see \cite{DS3}). However the Neumann condition for $w$ cannot be justified in this case. This is the main reason why we require initially $\Gamma \in C^{2,\alpha}$.
\end{rem}

\subsection{General setting}
Assume $u\in C(B_1)$ satisfies the thin one-phase problem with general free boundary condition
\begin{equation}\label{TOG}
\begin{cases}
\Delta u = 0 \quad \quad \text{in $B_1 \setminus \mathcal P$}, \quad \quad u=0 \quad \mbox{on $\mathcal P$,}\\
\frac{\partial u}{\partial U_0}=G(z) \quad \text{on $\Gamma$,}
\end{cases}
\end{equation}
with $G>0$, $G \in C^{k+2,\alpha}$.

If $\Gamma \in C^{k+2,\alpha}$ then the quotient $w=u_i/u_n$ satisfies a Neumann condition
$$w_\nu= h  \quad \quad \quad \mbox{on $\Gamma$},$$ for some $h \in C^{k+1,\alpha}$ depending on $\nu$, $G$ and the derivatives of $G$.

Indeed, as above, at the origin we find ($1 \le i \le n-1$)
$$a_0=G, \quad a_i=G_{z_i},$$
where $G$ and its derivatives are evaluated at $0$. Then on the line $te_n$ we obtain
$$(t^ \frac 12 u_i)(0)=0, \quad (t^\frac 12 u_i)_{e_n}(0)=G_{z_i} \quad \quad (t^ \frac 12 u_n)(0)=G(0)/2,$$
and now it is straightforward to obtain the dependence of $h$ on $\nu$, $G$, $\nabla G$.
Using Remark \ref{rf} we obtain optimal regularity of the free boundary in problem \eqref{TOG}.

\begin{prop}\label{thin0} Assume $u$ satisfies \eqref{TOG} for a positive $G \in C_x^{k+2,\alpha}$, for some $k \ge 0$ and $\alpha \in (0,1)$. If $\Gamma \in C^{2,\alpha}$ then $\Gamma \in C^{k+3,\alpha}$.
\end{prop}

\subsection{Constant coefficients}

 We prove our theorems using the estimates for the ``constant coefficients" case together with perturbation arguments. Precisely, Theorems \ref{Schauder} and \ref{thin} rely on the following two theorems.

\begin{thm} \label{SchauderL}
Assume $\Gamma=\{x_n=0\}$ and $u\in C(B_1)$ is even, $\|u\|\le 1$ and satisfies
$$\Delta u = 0 \quad \quad \text{in $B_1 \setminus \mathcal P$}, \quad \quad u=0 \quad \mbox{on $\mathcal P$.}$$
For any $k \ge 0$, there exists a polynomial $P_0(x,r)$ of degree $k$ such that $U_0 P_0$ is harmonic in $B_1 \setminus \mathcal P$ and
$$|u-U_0 P_0| \le C |X|^{k+1} U_0,$$
for some constant $C$ depending on $k$ and $n$.
\end{thm}

\begin{thm}\label{thin0L}
Assume $\Gamma=\{x_n=0\}$ and $w\in C(B_1)$, $\|w\| \le 1$ satisfies
\begin{equation}\label{w_con}
\Delta ((U_0)_n w)=0, \quad \mbox{in $B_1 \setminus P$}, \quad w_\nu=0 \quad \mbox{on $\Gamma$.}
\end{equation}

For any $k \ge 0$, there exists a polynomial $T(x,r)$ of degree $k$, of the form $$T=Q(x') + r  \, \, P(x,r),  \quad \quad \deg P=k-1,$$ such that $T$ satisfies \eqref{w_con} and
$$|w- T| \le C |X|^{k+1},$$
for some constant $C$ depending on $k$ and $n$.
\end{thm}

The proofs of these two theorems are postponed till the appendix. They use the linearity and the translation invariance in the $x'$ direction of the corresponding equations.

\section{Pointwise Schauder estimate}\label{s5}

In this section we present our key estimate, that is a pointwise Schauder estimate in slit domains. We prove it under rather general assumptions.   
Theorems \ref{Schauder} and \ref{exp} will easily follow from this result.

\begin{prop}[Pointwise Schauder estimate]\label{PSchauder}
Assume that $u\in C(B_1)$ is even and vanishes on $\mathcal P$, $\|u\|_{L^\infty} \le 1$, and
\begin{equation}\label{FB1}
\Delta u (X)= \frac{U_0}{r} \, R(x,r) +F(X) \quad \quad \quad \quad \mbox{in $B_1 \setminus \mathcal P$},
\end{equation}
with
$$|F(X)| \le r^{-\frac 12}|X|^{k+\alpha} \quad \quad \mbox{and $R(x,r)$ a polynomial of degree $k$ with $\|R\| \le 1$.}$$
 There exists a polynomial $P_0(x,r)$ of degree $k+1$ with coefficients bounded by $C$ such that
 $$\left|\frac{u}{U_0}-P_0\right| \le C |X|^{k+1+\alpha,}$$ and
 $$|\Delta(u-U_0P_0)| \le Cr^{-\frac 12}|X|^{k+\alpha} \quad \text{in $B_1 \setminus \mathcal P,$}$$ 
 with $C$ depending on $k$, $\alpha$, $n$.
\end{prop}

The proof of Proposition \ref{PSchauder} is similar to the proof of the classical pointwise Schauder estimates, but in our case we work with monomials $U_0 x^\mu r^\gamma$ instead of monomials of the type $x^\mu x_{n+1}^\gamma.$ The reason is that monomials $U_0 x^\mu r^\gamma$ remain of the ``same form" after applying $\Delta$.

Indeed, first notice that in a $2D$ plane $(t, x_{n+1})$ with $r=\sqrt{t^2+x_{n+1}^2}$ we have
$$\Delta_{t,x_{n+1}} (r^m U_0)= m(m+1)\,r^{m-2} U_0,$$
\begin{equation}\label{meq0}
\p_t(r^m U_0)= U_0(\frac{1}{2} r^{m-1} + m \, t r^{m-2}).
\end{equation}
Therefore in $\R^{n+1}$ we obtain
\begin{equation}\label{meq}
\Delta(r^m U_0) = m(m+1) \, r^{m-2}U_0 + \kappa(x) \, \p_t (r^m U_0)
\end{equation}
with $\kappa(x)$ the mean curvature of the parallel surface to $\Gamma$ passing through $x$. We also denote by $\nu(x)$ the normal to this parallel surface. Thus,
$$\kappa(x) = -\Delta d \in C^{k, \alpha}_x, \quad \nu(x)= \nabla d \in C^{k+1,\alpha}_x.$$

To fix ideas, we present the proof of Proposition \ref{PSchauder} first in the case $k=0$. Then we explain the general case.

\subsection{Proof of Proposition \ref{PSchauder} in the case $k=0$} We remark that in this case $R$ is a constant. After performing an initial dilation, we may assume that our hypotheses in $B_1$ are
$$\|\Gamma\|_{C^{2,\alpha}} \le \delta, \quad |R| \le \delta, \quad  |F| \le \delta r^{-\frac 12}|X|^{\alpha},$$
for some $\delta$ small, to be made precise later.

From the formulas above, we have

$$\Delta U_0 = \frac{1}{2} \kappa(x) \frac{U_0}{r},$$

$$\Delta(r U_0) = \left(2+(d+\frac{1}{2} r) \kappa(x)\right) \frac{U_0}{r},$$
and we easily compute

$$\Delta(x_i U_0) =  \nu^i  \frac{U_0}{r}.$$

If $$P(X)=a_0+\sum_{i=1}^n a_i x_i + a_{n+1} r,$$ then
\begin{equation}\label{DU}
\Delta(U_0 P) = \frac{U_0}{r} \left(\frac{\kappa(0)}{2} a_0 + a_n + 2 a_{n+1} + h_0(x) + r h_1(x)\right)
\end{equation}
with
$$h_0, h_1 \in C_x^\alpha, \quad \quad h_0(0)=0, \quad \quad \|h_0\|_{C^\alpha}, \|h_1\|_{C^\alpha} \leq C \delta \|P\|.$$

We say that $P$ is an {\it approximating polynomial} for equation \eqref{FB1} at 0, if
$$\frac{\kappa(0)}{2} a_0 + a_n + 2a_{n+1} = R.$$

We prove Proposition \ref{PSchauder} by approximating $u$ in a sequence of balls $B_{\rho^m}$ with appropriate functions $U_0 P_m$ with $P_m$ approximating polynomials.

It suffices to prove the next lemma.

\begin{lem}\label{Imp} There exist universal constants $\rho$, $\delta$ depending only on $\alpha$ and $n$, such that if $P$ with $\|P\| \le 1$ is an approximating polynomial for $u$ in $B_\lambda$, that is $P$ is approximating for \eqref{FB1} at 0 and $$|u-U_0P|_{L^\infty(B_\lambda)} \leq \lambda^{3/2+\alpha},$$ for some $\lambda>0$, then there exists an approximating polynomial $\bar P$ for $u$ in $B_{\rho \lambda}$:
$$|u -U_0\bar P| _{L^\infty(B_{\rho\lambda})} \leq (\rho\lambda)^{3/2+\alpha}, \quad \quad \|\bar P- P\|_{L^\infty(B_\lambda)} \le C \lambda^{1+\alpha}.$$
\end{lem}

\begin{proof} Define $\tilde u$ to be the error between $u$ and $U_0P$ rescaled at unit size, that is
$$u-U_0P=:\lambda^{\frac 3 2 +\alpha} \tilde u (\frac X \lambda).$$
Then our assumption reads $\|\tilde u\|_{L^\infty(B_1)} \leq 1.$
Since $u$ solves \eqref{FB1},  $$F+\frac{U_0}{r}R - \Delta (U_0P)=\lambda^{-\frac 1 2 +\alpha} \Delta \tilde u(\frac X \lambda),$$ thus using \eqref{DU} and that $P$ is an approximating polynomial we obtain
\begin{equation}\label{eq1}
 \Delta \tilde u(\frac X \lambda)= \lambda^{\frac 12 -\alpha}\left (F(X)-\frac{U_0}{r}(h_0(x)+rh_1(x)\right).
 \end{equation}
Using the hypothesis on $F$ we find
$$|\Delta \tilde u(X)| \leq C \delta r^{-\frac 12}  \quad \text{in $B_1$.}$$
Denote by $\tilde \Gamma$, $\tilde {\mathcal P}$, $\tilde U_0$ the rescalings of $\Gamma$, $\mathcal P$ and $U_0$ from $B_\lambda$ to $B_1$ i.e. $$\tilde \Gamma:=\frac{1}{\lambda} \Gamma, \quad \tilde{\mathcal P}: = \frac 1 \lambda \mathcal P, \quad \quad \tilde U_0(X) := \lambda^{-\frac 12} U_0(\lambda X).$$

We decompose $\tilde u$ as $$\tilde u= \tilde u_0 + \tilde v$$ with
$$\begin{cases}
\Delta \tilde u_0 =0 \quad \quad\text{in $B_1 \setminus \tilde{\mathcal P}$},\\
\tilde u_0 = \tilde u \quad \quad \quad \text {on $\p B_1 \cup \tilde{\mathcal P}$,}
\end{cases}$$
and
$$\begin{cases}
|\Delta \tilde v| \leq C \delta r^{-\frac 12}   \quad \quad \text{in $B_1 \setminus \tilde{\mathcal P}$},\\
\tilde v = 0 \quad \quad \quad \quad \quad \text {on $\p B_1 \cup \tilde{\mathcal P}.$}\\
\end{cases}$$

Using barriers we can show the following
\begin{equation}\label{claim}
\|\tilde v\|_{L^\infty(B_1)} \leq C\delta \tilde U_0.
\end{equation}
We postpone the proof of \eqref{claim} till later.

To estimate $\tilde u_0$ we observe that  $\tilde u_0$ is a harmonic function in $B_1 \setminus \tilde{\mathcal P},$ $|\tilde u_0| \leq 1$ and as $\delta \rightarrow 0$, $\tilde {\Gamma}$ converges in the $C^{2,\alpha}$ norm to the hyperplane $\{x_n=0\}.$ Moreover, $\tilde u_0$ is uniformly H\"older continuous in $B_{1/2}$. By compactness, if $\delta$ is sufficiently small universal, $\tilde u_0$ can be approximated in $B_{1/2}$ by a solution of the Laplace problem with $\Gamma=\{x_n=0\}$. Thus by Theorem \ref{SchauderL},
\begin{equation}\label{Q}
\|\tilde u_0 - \tilde U_0 Q\|_{L^\infty(B_\rho)} \leq C \rho^{2+\frac 1 2 }
\end{equation}
with $\|Q\| \le C$, and since $U_0Q$ is harmonic we also get from \eqref{DU} that
$$Q=b_0 + b_ix_i + b_{n+1}r, \quad 2b_{n+1}+b_n =0.$$
Using also \eqref{claim} we find
$$\|\tilde u - \tilde U_0 Q\|_{L^\infty(B_\rho)} \leq C \rho^{\frac 5 2 } +C\delta \leq \frac 1 2 \rho^{\frac 3 2 +\alpha}$$
provided that we choose first $\rho$ and then $\delta$, universal, sufficiently small.

Writing this inequality in terms of the original function $u$ we find,
$$|u-U_0(P+\lambda^{1+\alpha}Q(\frac X \lambda))| \leq \frac 1 2 (\lambda \rho)^{\frac 3 2 + \alpha} \quad \quad \mbox{in $B_{\rho\lambda}$}.$$
However $P(X)+\lambda^{1+\alpha}Q(X/ \lambda)$ is not an approximating polynomial and therefore we need to perturb $Q$ by a small amount.
Let $$\bar Q:= Q-\frac{1}{4} \kappa(0) b_0\lambda r $$ thus $P+ \lambda^{1+\alpha} \bar Q(X/ \lambda)$ is approximating. Notice that $$\|Q-\bar Q\| \leq C \delta$$ and therefore we can replace $Q$ by $\bar Q$ in \eqref{Q} and obtain the same conclusion.

We define
$$\bar P= P+ \lambda^{1+\alpha} \bar Q(\frac x \lambda),$$ thus
$$\|\bar P-P\|_{L^\infty(B_\lambda)}\leq C\lambda^{1+\alpha}.$$

This concludes the proof of the lemma.
\end{proof}

We can now conclude the proof of Proposition \ref{PSchauder}.

After multiplying $u$ by a small constant, we see that the hypotheses of the lemma are satisfied for some initial $\lambda_0$ small with $P=Rx_n$. Now we may iterate the lemma for all $\lambda=\lambda_0 \rho^m$ and conclude that there exists a limiting approximating polynomial $P_0$, $\|P_0\|\le C$, such that $$|u-U_0P_0|\leq C|X|^{\frac 3 2 +\alpha} \quad \quad \mbox{in $B_1$.}$$

In $B_\lambda$ we may argue as in the proof above with $P_0$ replacing $P$ and obtain $$|\tilde u| \leq |\tilde u_0 | + |\tilde v| \le C \tilde U_0 \quad \text{in $B_{1/2}$,}$$
where we have used boundary Harnack inequality for $\tilde u_0$ and \eqref{claim} for $\tilde v$.
Thus,
$$\|u-U_0P_0\|_{L^\infty(B_\lambda)} \leq C \lambda^{1+\alpha} U_0.$$
Moreover, since $P_0$ is approximating, by \eqref{DU}
$$\Delta (u-U_0P_0)= F(X) + \frac{U_0}{r}(h_0(x)+rh_1(x)) = O( r^{-\frac 12} |X|^{\alpha}).$$
We are left with the proof of \eqref{claim}.

\

{\it Proof of claim \eqref{claim}.} We use as lower (upper) barriers multiples of the function
$$\bar v:=-U_0+U_0^2.$$
Notice that $\bar v \le 0$ in $B_1$. In the 2D plane $(t,x_{n+1})$ we have
$$\Delta \bar v \ge 2|\nabla U_0|^2 \ge c r^{-1}, \quad \quad |\p_t \bar v| \le Cr^{-\frac 12},$$
thus in $\R^{n+1}$ we also satisfy
$$ \Delta \bar v \ge c r^{-1}.$$
\qed

\

 We present some remarks which we often use about functions $w \in C^{k,\alpha}_{xr}(\Gamma)$.
Assume for simplicity that $k=1$ since the general case follows similarly.

\begin{rem}\label{r1} Let $P_0$ and $P_Z$ be the tangent polynomials for $w$ at $0$ and $Z \in \Gamma$ with $|Z| =\lambda$. Since both $P_0$ and $P_Z$ approximate $w$ in $B_{\lambda /2}(\lambda e_n)$ with a $C \lambda^{1+\alpha}$ error, then
$$\|P_0-P_Z\|_{L^\infty(B_{2\lambda})} \leq C \lambda^{1+\alpha}$$
and this implies that the free coefficients of $P_0$ and $P_Z$ differ by $C|Z|^{1+\alpha}$ and the first order coefficients differ by $C|Z|^\alpha$.

For general $k$ we obtain that the corresponding coefficients of the monomials of degree $m$ for $P_0$ and $P_Z$ differ by $C|Z|^{k-m+\alpha}$.

Notice that we only used that $P_0$ (respectively $P_Z$) approximates $w$ in a cone around the corresponding normal to $\Gamma$, say $\{|X| \le x_n \}$.
\end{rem}

\begin{rem}\label{r1*} Let $W$ be the function $W(X)=P_Z(x,r)$ where $Z$ denotes the projection of $x$ onto $\Gamma$. In other words $W$ coincides with the tangent polynomial on each 2D plane perpendicular to $\Gamma$. Then
$$W(X)=a_0(Z)+a_n(Z) d+ a_{n+1}(Z) r, \quad \quad w= W + O( r^{1+\alpha}),$$
for some functions $a_0$, $a_n$, $a_{n+1}$ defined on $\Gamma$. Thus, $w$ and $W$ have the same tangent polynomials on $\Gamma$. Now it is not difficult to
show that $a_0 \in C^{1,\alpha}$, $a_n \in C^{\alpha}$, $a_{n+1} \in C^\alpha$.

For general $k$ we find that $W$ is a polynomial of degree $k$ in $(d,r)$ with coefficients depending on $Z$. The monomials of degree $m$ in $(d,r)$ have coefficients in $C^{k-m, \alpha}(\Gamma).$
\end{rem}

\subsection{Applications of Proposition \ref{PSchauder} and Proof of Theorem \ref{Schauder}.}

It is clear that the statement \eqref{Su} in Theorem \ref{Schauder} follows from the pointwise estimate in Proposition \ref{PSchauder} applied with $f(X)=R(x,r)+ h(X)$ with $R$ a polynomial of degree $k$ and $h(X)= O(|X^{k+\alpha}|).$ To obtain \eqref{SDu}, we need to deduce 
some consequences of Proposition \ref{PSchauder} in which we estimate the derivatives of $u$ near $\Gamma$. Roughly speaking we can estimate $\nabla u$ by differentiating formally the expansion of $u$. However in order to do this we need to impose slightly more regularity on the right hand side $F$. First we notice that, by scaling, we can estimate the derivatives of $u$ from the conclusion of Proposition \ref{PSchauder} in non-tangential cones to $\Gamma$.
\begin{lem}\label{lg}
Assume that $u$ satisfies the hypotheses of Proposition $\ref{PSchauder}$. Then
\begin{equation}\label{nu}
\left |u_i-\frac{U_0}{r}P^i_0\right | \le C|X|^{\frac 12 +\alpha+k} \quad \quad \mbox{in the cone} \quad \{r \ge |x'| \},
\end{equation}
with $\deg P_0^i=k+1$, and $(U_0/r)P^i_0$ is obtained by formally differentiating $U_0 P_0$ at the origin in the $x_i$ direction.
\end{lem}

\begin{rem}\label{r2}
If the hypotheses of Proposition \ref{PSchauder} are satisfied at all points $Z \in \Gamma \cap B_{1/2}$ instead of only the origin then we obtain that the inequality \eqref{nu} holds in fact for all $X$ in a neighborhood of the origin. This follows easily by applying the arguments of Remark \ref{r1} to $u_i$.
\end{rem}

\begin{proof} We assume $k=0$. As in the proof of Proposition \ref{PSchauder} denote by $\tilde u$ the rescaling of $u-U_0P_0$ from $B_\lambda$ to $B_1$ i.e.
$$u-U_0P_0=\lambda^{\frac 32+\alpha}\tilde u (X/\lambda),$$
thus
$$\Delta \tilde u = \tilde F, \quad \quad \|\tilde u\|_{L^\infty(B_1)} \le C,$$ with

\begin{equation}\label{tildeF}\tilde F(X) := \lambda^{\frac 12-\alpha}{F(\lambda X)} + \frac{\tilde U_0}{r} \left( \lambda^{-\alpha}h_0(\lambda x) + \lambda^{1-\alpha} r h_1(\lambda x) \right).\end{equation}Let $\mathcal C$ denote the conical domain
$$\mathcal C :=\{ r \ge 2|x'|\} \cap (B_1 \setminus B_{1/4}).$$
Then
$$\|\tilde F\|_{L^\infty(\mathcal C)} \leq C,$$
hence
\begin{equation}\label{nu1}
|\nabla_x \tilde u| \leq C  \quad \text{in $\mathcal C':=\{r \geq |x'|\} \cap (B_{3/4} \setminus B_{1/2})$.}
\end{equation}
This gives, for all $\lambda>0$
$$|\nabla_x(u-U_0P_0)| \leq C \lambda^{\frac 1 2 +\alpha}  \quad \text{in $\mathcal C':=\{r \geq |x'|\} \cap (B_{\frac 3 4 \lambda} \setminus B_{\frac 1 2 \lambda})$.}$$
On the other hand,
$$\nabla_x(U_0P_0) = \frac{U_0}{r} \left [ \frac 12 P_0\nu + r \nabla_xP_0 + (\p_r P_0) d\nu \right].$$ Since $\nu, d \in C_x^{1,\alpha}$ we obtain in $\{r \geq |x'|\}$
$$|\p_i(U_0P_0) - \frac{U_0}{r}[P_0^i(x,r)]| \leq C\frac{U_0}{r}|X|^{1+\alpha}$$ with $deg P_0^i=1,$
and this proves Lemma \ref{lg}.
\end{proof}

We present some variations of Lemma \ref{lg}, which will lead to the proof of the second part of Theorem \ref{Schauder} as well.

\

1) If $\lambda^{\frac 12-\alpha}F(\lambda X)$ is uniformly H\"older continuous at all points in the conical $n$-dimensional set $\mathcal P \cap \mathcal C$ then, since $h_0, h_1$ are H\"older continuous, $\tilde F$ is also uniformly H\"older continuous at all points in this set (see formula \eqref{tildeF}).
Then, since the $u_i$'s ($1 \leq i \leq n$) vanish on the  plate $\mathcal P$, we can improve \eqref{nu1} to
$$|\nabla_x u| \le C \tilde U_0  \quad \text{in $\mathcal C'$.}$$
This means that the right hand side in $\eqref{nu}$ can be replaced by $C|X|^\alpha U_0$, that is
\begin{equation}\label{nuimp}
\left |u_i-\frac{U_0}{r}P^i_0\right | \le C\frac{U_0}{r}|X|^{\alpha+1+k}\quad \quad \mbox{in the cone} \quad \{r \ge |x'| \}.
\end{equation}

It is easy to check that this is the case when $F$ has the form $(U_0/r)h$ with $h(0)=0$ and $h$ pointwise $C^\alpha_X$ at $0$. 

Now \eqref{SDu} in Theorem \ref{Schauder} readily follows from \eqref{nuimp}, by decomposing  $f(X)=R(x,r)+ h(X)$ with $R$ a polynomial of degree $k$ and $h(X)= O(|X^{k+\alpha}|)$  and arguing as in Remark \ref{r2}.

\

2) If $\lambda^{\frac 12-\alpha}F(\lambda X)$ is uniformly H\"older continuous at all points in $\mathcal C$ then we can estimate the second derivatives. Indeed (see \eqref{tildeF}), $\|\tilde F\|_{C^\alpha(\mathcal C)} \leq C$, thus
\begin{equation}\label{D2}|D^2_x\tilde u| \le C \quad \mbox{in $\mathcal C'$}.\end{equation}
Since
$$\p_{ij}(U_0P_0) = \frac{U_0}{r^3} (P_0^{ij} + O(|X|^{k+2+\alpha})), \quad 1 \leq deg P_0^{ij} \leq k+2,$$
we obtain
$$\left |u_{ij}-\frac{U_0}{r^3} P_0^{ij} \right | \le C |X|^{k+\alpha-\frac 12} \quad \quad \mbox{in the cone} \quad \{|x'| \le r\}. $$

In the case $F \equiv 0$ then we can improve this estimate. Indeed, $\tilde F$ is now pointwise $C^{1,\alpha}$ and \eqref{D2} can be replaced by \begin{equation}\label{D2*}|D^2_x\tilde u| \le C \tilde U_0 \quad \mbox{in $\mathcal C'$}.\end{equation}

Then, arguing as in part 1) we obtain
$$ \left |u_{ij}-\frac{U_0}{r^3} P_0^{ij}\right | \le C \frac{U_0}{r^3}|X|^{k+1+ \alpha}.$$

\section{The proof of Proposition \ref{PSchauder}. The general case.}\label{s6}

The proof is essentially the same as in the case $k=0$. The difference occurs in the notion of approximate polynomial, since we need to satisfy several linear equations rather than just a single one.

We now proceed to give the definition of approximating polynomial, for this general case.

Let $\bar i $ denote the vector of indices $\mu$ with $1$ on the $i$th position and zeros elsewhere.
Using \eqref{meq0}, \eqref{meq}, we obtain
\begin{align*}
\Delta(x^\mu r^{m} U_0) &= r^m U_0 \Delta(x^\mu) + x^\mu \Delta(r^m U_0)+ 2 \nabla x^\mu \cdot \nabla (r^m U_0) \\
\  & = U_0  (r^m \mu_i (\mu_i-1)\, x^{\mu - 2\bar i} + m(m+1)\, x^\mu r^{m-2} +  \\
\   & + x^\mu(\frac 1 2 r^{m-1} + m \,d r^{m -2}) \kappa(x) + 2(\frac 1 2 r^{m-1} + m\, d r^{m-2}) \nu \cdot \nabla_x x^\mu ).
\end{align*}

By Taylor expansion at 0, we write each $\nu^i, d$ and $\kappa$ as a sum between a polynomial of degree $k$ and a $C^{k,\alpha}$ function in $x$ with vanishing derivatives up to order $k$ at 0. We use that the lowest degree terms in each expansion are
\begin{equation}\label{low} \nu^i= \delta_n^i + \ldots, \quad \kappa= \kappa(0)+\ldots, \quad d=x_n+\ldots
\end{equation}
We arrange the terms in $\Delta (x^\mu r^m U_0)$ by  the degree up to order $k$ an group the remaining ones in a remainder. Precisely, 
\begin{align*} \Delta(x^\mu r^m U_0)& = \frac{U_0}{r} [m(m+1+2\mu_n) \, x^\mu r^{m-1} + \mu_n \, x^{\mu -\bar n} r^m+ \\
&  +\mu_i(\mu_i-1)\,x^{\mu-2\bar i} r^{m+1}  +  c_{\sigma l}^{\mu m} x^\sigma r^l + w^{\mu m} (x,r)],
\end{align*}
with $$c_{\sigma l}^{\mu m} \ne 0 \quad \mbox{only if} \quad \quad |\mu| + m-1<|\sigma|+l \le k,$$
and $$w^{\mu m}(x,r) = r^m w^\mu_m(x) + m \, r^{m-1}w^\mu_{m-1}(x)$$ with $w^\mu_m$ and $w^\mu_{m-1}$ of class $C^{k,\alpha}_x$ with vanishing derivatives of all orders up to $k-m$ respectively $k-(m-1)$ at 0.

The monomials $c_{\sigma l}^{\mu m} x^\sigma r^l$ have strictly higher degree than the first terms and together with $w^{\mu m}$ can be thought as lower order terms. Notice that the coefficients $c_{\sigma l}^{\mu m}$ are linear combinations of polynomial coefficients at 0 of $\kappa(x), d \kappa (x), \nu^i, d \nu^i$ from \eqref{low}, and they vanish in the flat case $\Gamma=\{x_n=0\}$.

Thus under the assumption $\|\Gamma\|_{C^{k+2, \alpha}} \leq \delta$ (achieved after a dilation), 
 we may suppose that
\begin{equation}\label{small}|c_{\sigma l}^{\mu m}|\le \delta, \quad \|w^\mu_m\|_{C^{k,\alpha}},  \|w^\mu_{m-1}\|_{C^{k,\alpha}} \leq \delta.\end{equation}

If $$P= a_{\mu m} x^\mu r^m \quad \mbox{is a polynomial of degree $k+1$,}$$
then
$$\Delta (U_0P)= \frac{U_0}{r} (A_{\sigma l} x^\sigma r^l + w(x,r)), \quad \quad \quad |\sigma|+l \leq k,$$
with
\begin{align}\label{A} A_{\sigma l} & =  (l+1)(l+2 + 2\sigma_n) \, a_{\sigma,l+1}  + (\sigma_n +1) a_{\sigma +\bar n,l} +\\
 \nonumber &+ (\sigma_i +1)(\sigma_i+2) a_{\sigma+2\bar i,l-1} + c^{\mu m}_{\sigma l} a_{\mu m},
 \end{align}
 and
 $$w(x,r)= \sum_{m=0}^k r^m w_m(x),$$ with $w_m \in C_x^{k,\alpha}$ and derivatives up to order $k-m$ vanishing at zero. 
 Again, under the assumption $\|\Gamma\|_{C^{k+2, \alpha}} \leq \delta$, we have
 $$\|w_m\|_{C^{k,\alpha}} \leq \delta \max |a_{\mu \gamma}|.$$

From \eqref{A} we see that $a_{\sigma, l+1}$ (whose coefficient is different than 0) can be expressed in terms of $A_{\sigma l}$ and a linear combination of $a_{\mu m}$ with $ \mu + m < |\sigma| + l +1$
plus a linear combination of $a_{\mu m}$ with $ \mu + m = |\sigma| + l +1$ and $m < l+1$. This shows that the coefficients $a_{\mu m}$ are uniquely determined from the linear system \eqref{A} once $A_{\sigma l}$ and $a_{\mu 0}$ are given.

 \begin{defn}\label{d1} We say that $P$ is approximating for the equation \eqref{FB1} if $A_{\sigma l}$ coincide with the coefficients of $R$.
 \end{defn}

To obtain the proof of Proposition \ref{PSchauder} in the general case, is now enough to
obtain the same improvement of flatness as in Lemma \ref{Imp}, with the approximating polynomials defined above

Indeed, assume that (after a dilation)
$$\|\Gamma\|_{C^{k+2,\alpha}} \le \delta, \quad |R| \le \delta, \quad  |F| \le \delta r^{-\frac 12}|X|^{\alpha}.$$
Since $P$ is approximating, arguing as in the case $k=0$ we have
 $$u-U_0P=:\lambda^{k+\frac 3 2 +\alpha} \tilde u (\frac X \lambda),$$
and
\begin{equation}\label{Ft}
 \Delta \tilde u(\frac X \lambda)= \lambda^{\frac 12 -k-\alpha}\left (F(X)-\frac{U_0}{r}w(x,r)\right)=:\tilde F(\frac X \lambda).
 \end{equation}
Using that $\lambda^{m-k} w_m(\lambda x)$ has bounded $C^{k,\alpha}_x$ norm in $B_1$ together with the hypothesis on $F$, we obtain $$\|\tilde F(X)\|_{L^\infty(B_1)} \leq \delta.$$

Now the proof is the same as Lemma \ref{Imp}. The only difference is that the approximating polynomial $Q$ has degree $k+1$ and satisfies (see \eqref{A})
$$(l+1)(l+2 + 2\sigma_n) \,  q _{\sigma,l+1}  + (\sigma_n +1) \, q _{\sigma +\bar n,l} + (\sigma_i +1)(\sigma_i+2) \, q_{\sigma+2\bar i,l-1} =0,$$
with bounded $q_{\mu m}$.
 Then we need to modify $Q$ into $\bar Q$ such that $\bar Q(x/\lambda, r/\lambda)$ is approximating for $R \equiv 0$. Thus its coefficients solve the system \eqref{A} with $A_{\sigma l}=0$ and rescaled $c^{\mu m}_{\sigma l}$, i.e.
\begin{align}\label {barQ}   (l+1)(l+2 + 2\sigma_n) \, \bar q_{\sigma,l+1}  + (\sigma_n +1) \bar q_{\sigma +\bar n,l} +&  \\
 \nonumber + (\sigma_i +1)(\sigma_i+2) \bar q_{\sigma+2\bar i,l-1} +  \bar c^{\mu m}_{\sigma l} \bar q _{\mu m}&=0,
 \end{align}
 with
 $$\bar c^{\mu m}_{\sigma l}:=\lambda^{|\sigma|+l+1-|\mu|-m}c^{\mu m}_{\sigma l}, \quad \quad \mbox{hence} \quad |\bar c^{\mu m}_{\sigma l}|\le |c^{\mu m}_{\sigma l}| \le \delta.$$
 After subtracting the last 2 equalities we see that the coefficients  of $Q-\bar Q$ solve the linear system \eqref{barQ} with right hand side
 $A_{\sigma l} = \bar c^{\mu m}_{\sigma l}  q _{\mu m} $, hence $|A_{\sigma l}| \le C \delta$. 
 Thus by choosing $\bar q_{\mu 0}=q_{\mu 0}$ we can solve uniquely for $\bar Q$ and find
 $$\|\bar Q-Q\| \leq C\delta.$$
 
 This concludes the proof.
  \qed

\

Now we sketch the proof of Theorem \ref{exp} which follows from Proposition \ref{PSchauder}.

\

{\it Proof of Theorem \ref{exp}.} We assume that
$$\Delta u=F \quad \mbox{in} \quad B_1 \setminus \mathcal P, \quad \quad F \in C^{k, \alpha}(B_1),$$
and $u$ vanishes continuously on $\mathcal P$, and $\Gamma=\p_{\R^n} \mathcal P \in C^{k+2}$.
We decompose $u=u_{ev} + u_{od}$ in the even and odd part which solve the Laplace equation with right hand side $F_{ev}$ respectively $F_{od}$.

We have the following expansions
$$u_{od}=x_{n+1}\left (P_{od}(x,x_{n+1})+O(|X|^{k+1}) \right), \quad \quad \deg P_{od}=k,$$
for some polynomial $P_{od}$, even in $x_{n+1}$. For the even part we can write
$$\bar u(X):=u_{ev}(X)-x_{n+1}^2 T_0(x,x_{n+1}), \quad \quad \deg T=k,$$
for some appropriate even polynomial $T_0$, such that
$$\Delta \bar u=\bar F \quad \mbox{with} \quad \bar F=O(|X|^{k+\alpha}).$$
 We can apply for $\bar u$ Proposition \ref{PSchauder} at the origin and obtain
 $$\bar u=U_0(P_0(x,r) + O(|X|^{k+1+\alpha})).$$
 In conclusion
 $$u=U_0 P_0 + x_{n+1} P_{od} + x_{n+1}^2 T_0 + O( U_0|X|^{k+1+\alpha})).$$
Writing this at all points on $\Gamma$ and using the arguments of Remark \ref{r1} we obtain
$$u=U_0 \left(\sum_{m=0}^{k+1}a_m(x)r^m +O(r^{k+1+\alpha})\right) + x_{n+1}\left (\sum_{m =0}^k b_m(x) x_{n+1}^m \right ) ,$$
for functions $$a_m \in C^{k+1-m,\alpha}, \quad b_m \in C^{k-m, \alpha +\frac12}.$$
\qed

\

We conclude this section with the estimates for the derivatives of harmonic functions in slit domains.

\begin{prop}\label{ph}
Assume $\Gamma \in C^{k+2,\alpha}$, $k \ge 0$, and $u \in C(B_1)$, even, satisfies
$$\Delta u =0 \quad \mbox{in} \quad B_1 \setminus \mathcal P, \quad \quad u=0 \quad \mbox{on} \quad \mathcal P.$$
If $|\mu| \le k+2$, then
$$ D_x^\mu u = \dfrac{U_0}{r^{2|\mu| -1}}(P_0^{\mu}(x,r) + O(|X|^{k+|\mu|+\alpha})),$$ with $P_0^\mu$ having monomials with degrees between $|\mu| -1$ and $k+|\mu|.$

Moreover, $U_0r^{1-2|\mu|} P_0^\mu$ is obtained by differentiating formally $D^\mu_x (U_0 P_0)$ at $0$.
\end{prop}

 Indeed, as we discussed in Lemma \ref{lg} in the case $k=0$, these estimates follow from the proof of Proposition \ref{Schauder}. From \eqref{Ft} we see that $\tilde F$ is pointwise $C^{k+1,\alpha}$ on the set $\mathcal P \cap \mathcal C$. Hence if $|\mu| \le k+2$,
 $$|D_x^\mu \tilde u| \le C U_0 \quad \mbox{in} \quad \mathcal C',$$
which gives the conclusion in the cone $\{r \ge |x'|\}$. However, this expansion is valid around each such cone centered on $\Gamma$.
Then as in Remark \ref{r2} we can show that the conclusion holds in fact in a whole neighborhood of $0$.

\section{Proof of Theorem \ref{thin}}

In this section we prove Theorem \ref{thin}. We assume throughout that $$ \Gamma \in C^{k+2,\alpha}, \quad \quad \|\Gamma \|_{C^{k+2,\alpha}} \le \delta,$$
$$\Delta u=0 \quad \mbox{in $B_1 \setminus \mathcal P$,} \quad \quad \mbox{ $u$ is even and} \quad \frac 12 U_0 \le u \le 2 U_0,$$
and $w \in C(B_1)$, even, $\|w\|_{L^\infty} \leq 1$ solves
\begin{equation}\label{NP}
\begin{cases}
\Delta(u_n w) = 0 \quad \text{in $B_1 \setminus \mathcal P,$}\\
w_{\nu} = 0 \quad \text{on $\Gamma \cap B_1.$}
\end{cases}
\end{equation}

We want to show that $w \in C_{xr}^{k+2, \alpha}(0)$, that is we can find a polynomial $T_0(x,r)$, $\deg T_0=k+2$, such that $$|w-T_0| \leq C|X|^{k+2+\alpha},$$
with $C$ depending on $n,k, \alpha$. Throughout this section we use $O(|X|^\beta)$ as a notation for functions that are bounded by $C|X|^\beta$ with $C$ depending only on $n,k, \alpha$.

The proof follows the lines of the proof of Proposition \ref{PSchauder}, however it is more technical since it involves the singular weight $u_n$. This time we do not approximate directly $w$ by a polynomial of degree $k+2$, but rather by a sum between $r P(x,r)$, with $\deg P=k+1$, and a $C_x^{k+2,\alpha}$ function of $x$ with vanishing normal derivative on $\Gamma$. This function of $x$ has also the property that it solves \eqref{NP} with a controlled right hand side. A polynomial $Q$ of degree $k+2$ in $x$ does not have these properties, and we need to adjust it in order to satisfy them. Next we construct such functions.

\subsection{Definition of $E(Q)$}
Let $$ y \rightarrow x=(y', g(y')) + y_n \nu, \quad \quad \nu:=\frac{(-\nabla g,1)}{\sqrt{1+|\nabla g|^2}}$$ be a change of coordinates from $\R^n$ to $\R^n$ which maps $\{y_n=0\}$ to $\Gamma$ and the lines in the $y_n$-direction into lines perpendicular to $\Gamma$. Since $g \in C^{k+2,\alpha}$ this change of coordinates is of class $C^{k+1,\alpha}$, and at least formally it is of class $C^{k+2,\alpha}$ pointwise on $\{y_n=0\}$.
Let
\begin{equation}\label{dQ}
Q=Q(y') = q_{\mu} y^{\mu}, \quad \quad \quad |\mu| \leq k+2, \quad \quad q_{\mu}=0 \quad \mbox{if $\mu_n \ne 0$,}
\end{equation}
be a $k+2$ polynomial in $y'$ (hence it is constant in the $y_n$-direction).

We work with such polynomials $Q$ viewed as functions of the $x$-variable. As a function of $x$, $Q$ is only a $C^{k+1,\alpha}$ function. However, we show below that on $\Gamma$, $Q$ it is pointwise $C^{k+2,\alpha},$ that is it can be approximated by a polynomial of degree $k+2$ in $x$ with an error of order $k+2+\alpha$.

\

{\it Claim:} $Q \in C_x^{k+2,\alpha} (\Gamma).$

{\it Proof:} It suffices to show that each coordinate function $y_i$, $1\le i \le n$, is in $C_x^{k+2,\alpha} (\Gamma)$. We show this at the origin by comparing  the corresponding coordinate functions $y_i$ (as functions of $x$) for $\Gamma$ and $\Gamma_t$ its tangent $k+2$ polynomial at the origin. The coordinate functions $y_i$ differ in $B_\rho$ by $C \rho^{k+2+\alpha}$ since $\Gamma$ and $\Gamma_t$ differ by $C \rho^{k+2+\alpha}$ and the normals $\nu_\Gamma$ and $\nu_{\Gamma_t}$ differ
by $C \rho^{k+1+\alpha}.$ Clearly, in the case of $\Gamma_t$, $y_i$ is $C_x^{k+2,\alpha}$ at the origin. Thus the same holds for $\Gamma.$
\qed

\

Now we extend (regularize) $Q$ away from $\Gamma$ without changing its $k+2$ tangent polynomials on $\Gamma.$ The extension $E(Q)$ has the property that it is of class $C^{k+2,\alpha}$ at all $x$'s and on $\Gamma$ it coincides with $Q$ up to order $k+2$. The existence of $E(Q)$ follows from Whitney's extension theorem (see for example \cite{F}). For completeness,  we present its simple proof for our case in the Appendix.
Precisely, we have the following theorem.

\begin{thm}[Whitney Extension Theorem]\label{WE}
There exists $E(Q)$ such that $$D^\mu_x E(Q) = D^\mu_x Q \quad \text{on $\Gamma$, for all $|\mu| \leq k+2$}$$ and
$$\|E(Q)\|_{C_x^{k+2,\alpha}(B_1)} \leq C \|Q\|_{C_x^{k+2,\alpha} (\Gamma)}.$$
Moreover $E(Q)$ is linear in $Q$, and if $Q$ is given by \eqref{dQ}, then
$$E(Q)= \tilde q_\mu x^{\mu} + O(|x|^{k+2+\alpha})$$ with $$\tilde q_\mu = q_\mu + \tilde c_\mu^\sigma q_\sigma, \quad \quad \tilde c_\mu^\sigma \ne 0 \quad \mbox{only if $|\sigma| < |\mu|.$}$$
\end{thm}

The last claim follows from the fact that $E(Q)$ and $Q$ have the same tangent polynomial at $0$ and for $1 \le i \le n-1$ we write $y_i$ as a polynomial of degree $k+2$ in $x$ plus an error $O(|x|^{k+2+\alpha})$. The first order in each expansion is
$$y_i=x_i + \, \, \mbox{\it{lower order terms}},$$
and $\tilde c_\mu^\sigma$ depend on the derivatives up to order $k+2$ of $g$ at the origin.

In the proof of Theorem \ref{thin}, we approximate $w$ at the origin by the sum of $E(Q)$ for some $Q$ as in \eqref{dQ} and a function in $C_{xr}^{k+2,\alpha}(0)$.

\subsection{Properties of $E(Q)$}
First we notice that, since $Q$ is constant on perpendicular lines to $\Gamma$ then $Q_\nu=0$ on $\Gamma$. Thus
$$E(Q)_\nu=0 \quad \mbox{on $\Gamma$.}  $$

In the next lemma we estimate $\Delta (u_e E(Q))$ for some unit vector $e$.

\begin{lem}\label{Ma} Let $e$ be a unit vector and let $u_e$ have the following expansion at $0$,$$u_e= \frac{U_0}{r} (P^e_0 + O(|X|^{k+1+\alpha})), \quad \quad \deg P_0^e=k+1.$$ Then, $$\Delta(u_e E(Q)) = \frac{U_0}{r} \left(R + O(|X|^{k+\alpha}) \right) \quad \text{in $B_1 \setminus \mathcal P,$}$$ with $R$ a polynomial of degree $k$ in $(x,r)$ and
$$R= A_{\sigma l} \,x^\sigma r^l , \quad \quad |\sigma| + l \leq k$$
with
$$A_{\sigma l} =
\begin{cases}
c_{\sigma l}^{\mu} q_{\mu}, \quad \quad  \quad \text{if $(\sigma_n,l) \neq (0,0)$,}\\
P^e_0(0) \,(\sigma_i + 1)(\sigma_i+2)\, q_{\sigma + 2\bar i}  + c_{\sigma l}^{\mu} q_{\mu} \quad \text{if $(\sigma_n,l)=(0,0)$,}
\end{cases}
$$
and $$c_{\sigma l}^{\mu} \ne 0 \quad \mbox{only if} \quad   |\mu| \leq |\sigma| + l +1.$$
\end{lem}
\begin{proof} Since $\Delta u_e=0$ we have
$$\Delta (u_e E(Q)) = u_e \Delta \, E(Q) + 2\nabla u_e \cdot \nabla \, E(Q).$$

From Theorem \ref{WE} we know that $\Delta \, E(Q)$ is pointwise $C_x^{k,\alpha}$ at the origin and its expansion is obtained by formally differentiating the expansion of $E(Q)$ at the origin. Next we estimate the second term by making use that $\nabla E(Q)$ is almost parallel to $\Gamma$.

We claim that $$\nabla E(Q) = \nabla Q + |d|^{k+\alpha} \xi + |d|^{k+\alpha+1} \eta$$ for two bounded vectors $\xi, \eta \in \R^n$ with $\xi \cdot \nu =0.$

Assume for simplicity that $x$ is a point on the $e_n$-axis. Then, since $E(Q) \in C_x^{k+2,\alpha}$ we find
$$\nabla E(Q)(x) = \sum_{m=0}^{k+1}\frac{1}{m!} x_n^{m} \nabla \partial_n^mQ(0) + O(|x_n|^{k+1+\alpha}).$$
Since $Q \in C_x^{k+1,\alpha}$ we see by Taylor expansion that $$ \sum_{m=0}^{k+1}\frac{1}{m!} x_n^{m} \nabla \partial_n^mQ(0) =\nabla Q(x) + \xi |d|^{k+\alpha},$$ for some bounded vector $\xi$.
Moreover, since $Q$ is constant on perpendicular lines to $\Gamma$, $\nabla Q(x) \cdot e_n=0$ and 
$\p^l_nQ(0) =0$ for all $l \leq k+2$. Thus, the formula above gives  $\xi \cdot e_n=0$ and our claim is proved.

From Proposition \ref{ph}
$$u_{ei} = \frac{U_0}{r^2} \left [(\frac 1 2 - \frac d r) \nu^i P^e_0 + r \, \, \p_iP^e_0 + d \, \nu^i \, \p_rP^e_0  + O(|X|^{k+1+\alpha}) \right].$$
Thus at a point $X$ in the 2D plane perpendicular to $\Gamma$ at $0$, i.e. $\{x'=0\}$, we have $\xi \cdot e_n=0$, $\nabla Q \cdot e_n=0$, $\nu^i=0$ for $i \ne n$, hence
$$|\xi \cdot \nabla u_e| \leq C \frac{U_0}{r}, \quad |\eta \cdot \nabla u_e| \leq C \frac{U_0}{r^2}.$$
and
$$\nabla u_e \cdot \nabla Q = \frac{U_0}{r} [\p_i P^e_0 \, \,  \p_i Q +O(r^{k+\alpha})].$$

This means that at an arbitrary point $X$ we find
$$\nabla u_e \cdot \nabla E(Q) = \frac{U_0}{r}[\partial_x P^e_Z \cdot \nabla Q + O(r^{k+\alpha})],$$ where $P_Z^e$ is the $k+1$ order polynomial in the expansion of $u_e$ at $Z \in  \Gamma$, projection of $X$ onto $\Gamma$. Also, for a polynomial $P(x,r)$, $\p_x P(x,r)$ denotes the gradient with respect to $x$ with $r$ thought as independent of $x$.

As in Remark \ref{r1} we may replace $ \p _x P^e_Z$ with $\p _x P^e_0$ and create an error of order $O(|Z|^{k+\alpha}) $. In conclusion
$$\nabla u_e \cdot \nabla \, E(Q) = \frac{U_0}{r} (\p_x P^e_0 \cdot \nabla Q + O(|X|^{k+\alpha})),$$ and
$$\Delta(u_e  \, E(Q))= \frac{U_0}{r} [P^e_0 \Delta \, E(Q) + \p_x P^e_0 \cdot \nabla Q + O(|X|^{k+\alpha})],$$
and the conclusion follows by using the expansions for $E(Q)$ and $Q$ at the origin.
\end{proof}

We remark that the coefficients $c_{\sigma l}^\mu$ depend on the coefficients of $P_0^e$ and $\tilde c_{\mu}^\sigma$.

\subsection{Compactness of solutions to \eqref{NP}} Let $P_0$ be the approximating polynomial for $u$ at 0 (given by Theorem \ref{Schauder}).
From now we assume, after multiplying $u$ by a constant (recall that $ \frac 12 U_0 \le u \le 2 U_0$), that $P_0(0)=1$ thus
 \begin{equation}\label{ua}
 u=U_0 (1+O(|X|)),
 \end{equation}
 and then we find (see \eqref{form}) \begin{equation}\label{uan}u_n=\frac{U_0}{r}(\frac 12 + O(|X|)).\end{equation}

Notice that the rescalings $$\tilde u (X) := \lambda^{-1/2} u(\lambda X)$$ satisfy the same properties and as $\lambda \to 0$ and $$\tilde u \to U_0, \quad \tilde u_n  \to (U_0)_n, \quad \tilde \Gamma \to L:=\{ x_n=0, x_{n+1}=0\}.$$

Next we prove that $w$ is uniformly H\"older continuous. We prove this under more relaxed hypotheses on $w$.

\begin{lem} \label{hol}Let $u$ be as above and let $w$ satisfy
$$\begin{cases} |\Delta(u_n w)| \leq \dfrac{U_0}{r}  \quad \text{in $B_1 \setminus \mathcal P$}\\
 |w_{\nu}| \leq 1 \quad \text{on $\Gamma$},  \quad \quad \|w\|_{L^\infty(B_1)} \leq 1.\end{cases}$$ Then, $w \in C^\beta$ and $\|w\|_{C^\beta(B_{1/2})} \leq C,$ for some $\beta$ small, universal.
\end{lem}
\begin{proof} The fact that $w \in C^{\beta}$ away from $\Gamma$ is obvious.  We only need to show that the oscillation of $w$ as we approach $\Gamma$ decreases at a geometric rate.

The rescalings $$\tilde w(x) = w(\lambda x)$$ satisfy in $B_1$
$$\begin{cases} |\Delta(\tilde u_n \tilde w)| \leq \lambda^2 \dfrac{\tilde U_0}{r} \quad \text{in $B_1 \setminus \mathcal P$}\\
 |\tilde w_{\nu}| \leq \lambda \quad \text{on $\Gamma$}, \|\tilde w\|_{L^\infty(B_1)} \leq 1.\end{cases}$$
Thus by scaling, it suffices to show that if
$$\begin{cases} |\Delta(u_n w)| \leq \delta_0 \dfrac{U_0}{r} \quad \text{in $B_1 \setminus \mathcal P$}\\
 |w_{\nu}| \leq \delta_0 \quad \text{on $\Gamma$}, \|w\|_{L^\infty(B_1)} \leq 1,\end{cases}$$
with $|u- U_0| \leq \delta_0, \|g\|_{C^{2,\alpha}} \leq \delta_0$, then $$\text{osc}_{B_{\delta_0}} w \leq 2-\delta_0,$$ for some $\delta_0$ universal.

Assume $w(\frac 1 2 e_n) \geq 0.$
We construct a lower barrier for $w$ defined as
$$v: = -1+\delta_1 \left(\frac 1 4 + E(Q) + \frac{U_0}{u_n}(1+Mr)\right)$$
with $Q(y') = -|y'|^2$, and some $\delta_1$ small, and $M$ large to be made precise later. From Lemma \ref{Ma}
$$\Delta(u_n \, E(Q)) \geq -C\frac{U_0}{r},$$
and we first choose $M$ large such that (see \eqref{meq})
$$\Delta(U_0(1+Mr)) \geq (cM-C) \frac{U_0}{r} \ge 2C \frac{U_0}{r}.$$
Notice that on $\Gamma$, (see \eqref{uan})
$$\p_\nu \left(E(Q) +\frac{U_0}{u_n}(1+Mr)\right) \geq c.$$

We compare $w$ and $v$ in the  cylindrical region $ B_{3/4} \cap \{ r < c\}.$  We have
$$ v <-1\le w \quad \quad \mbox{on} \quad \p B_{3/4} \cap \{r<c\},$$
provided that we take $c$ sufficiently small.

Since $w(\frac 1 2 e_n) \geq 0,$ by Harnack inequality (and boundary Harnack) for $u_n w$ we obtain that $w \geq -1+c_0$ on $B_{3/4} \cap \{r=c\}$.
If we choose $\delta_1$ small enough, then we have
$$v \leq w \quad \text{on $B_{3/4} \cap \{r=c\}$}.$$

If $\delta_0 \ll \delta_1$ we have $\Delta (u_n w) \le \Delta (u_n v)$ in $(B_{3/4} \cap \{r<c\}) \setminus \mathcal P$. Then $ v \le w$ by the maximum principle and the conclusion easily follows since $v \ge -1+\delta_1/8$ in a neighborhood of $0$. Indeed, the minimum of $w-v$ cannot occur on $\Gamma$ because of the free boundary condition, cannot occur on $\mathcal P$ because of Hopf lemma and cannot occur in the interior because of the classical maximum principle.
\end{proof}

\begin{lem}[Compactness]\label{comp}
 Let $u_k (x)$ be a sequence of harmonic functions in $B_1 \setminus \mathcal{P}_k$, vanishing on $\mathcal{P}_k$, with \begin{equation*}
 u_k=U_0 (1+\delta_k O(|X|)) \quad \|\Gamma_k\|_{C^{2,\alpha}} \leq \delta_k,
 \end{equation*}
for a sequence $\delta_k \to 0.$
 Let $w_k$ satisfy
$$\begin{cases} |\Delta((\p_nu_k)w_k)| \leq \delta_k \dfrac{U_0}{r} , \quad \text{on $B_1 \setminus \mathcal P$}\\
 |\p_{\nu}w_{k}| \leq \delta_k \quad \text{on $\Gamma$}, \quad \quad \|w_k\|_{L^\infty(B_2)} \leq 1.\end{cases}$$
Then there is a subsequence of $w_k$ that converges uniformly on compact sets to $\bar w$ that satisfies the limiting equation in the flat case i.e.
$\bar{\mathcal P}= \{x_n <0\}, \bar \Gamma =\{x_n=0\}$, 
$$\begin{cases} \Delta((U_0)_n \bar w) = 0 \quad \text{in $B_1\setminus \bar {\mathcal P}$}\\
\p_\nu \bar w =0 \quad \text{on $\bar \Gamma$.}
\end{cases}$$
\end{lem}

\begin{rem} The free boundary condition for $\bar w$ is understood in the viscosity sense defined in \cite{DS1}, i.e. $\bar w$ cannot be touched on $L$ say at 0 by below by a function of the form
$$b-a_1|y'-y'_0|^2 + a_2r \quad \quad \mbox{with $a_1>0$ and  $a_2>0$},$$
for some constants $b_0$, $a_1$, $a_2$ and some vector $y_0'$.
\end{rem}
\begin{proof} The fact that $w_k \to \bar w$ (up to a subsequence) on compact subsets of $B_1$ follows from Lemma \ref{hol}. Also, from our assumptions $u_k \to U_0, \Gamma_k \to \bar  \Gamma.$

Clearly, $(U_0)_n \bar w$ is harmonic in the interior.
It remains to check the condition on $\bar \Gamma$, that is we cannot touch $\bar w$ by below with a function as above. Otherwise, we can also touch by below $\bar w$ strictly in a neighborhood of 0 with the function
$$b_0-2a_1|y'-y_1'|^2+ \frac {a_2}{ 2} r + Mr^2\quad \text{with $M \gg 2a_1$,}$$
for some $b_0$ and $y_1'$. Since $w_k \to \bar w$ uniformly, then we can touch $w_k$ by $$v_k:=b_k -E(2a_1|y'-y_1'|^2) + \frac{U_0}{u_n}(\frac {a_2} {4}+\frac{M}{2}r),$$ for some constant $b_k$.
As in the proof of Lemma \ref{hol}, $v_k$ is a strict subsolution to our Neumann problem for all $k$ large and we reach a contradiction.
\end{proof}

\subsection{Proof of Theorem \ref{thin}} We argue similarly to the proof of Proposition \ref{PSchauder} and we approximate $w$ inductively in sequence of balls $B_{\rho^m}$. However, in this case we do not use directly polynomials of degree $k+2$ in $(x,r)$, but rather functions which are pointwise $C^{k+2,\alpha}_{xr}(0)$ and approximate better the Neumann problem. Precisely we use functions of the type
 $$W_{Q,P}:=E(Q)+ \frac{U_0}{u_n} P$$
 with $Q$ a polynomial of degree $k+2$ in $y$ as in \eqref{dQ} i.e.
 $$Q=q_\mu y^{\mu}, \quad |\mu| \leq k+2, \quad \quad q_\mu=0 \quad \mbox{if $\mu_n \ne 0$,}$$
 and $P$ a polynomial of degree $k+1$ in $(x,r)$,
$$P= a_{\mu m} x^\mu r^m, \quad |\mu| + m \leq k+1.$$
By Theorem \ref{Schauder} we have
\begin{equation}\label{pst}
\frac{U_0}{u_n} = r \left (P_*(x,r) + O( |X|^{k+1+\alpha}) \right), \quad \quad \deg P_*=k+1,
\end{equation}
and also by \eqref{ua}, $P_*(0,0)=2$.

We say that a pair $(Q,P)$ is approximating for the Neumann problem \eqref{NP} if:

\

$(i) \quad P_*P$ vanishes of order $k+1$ on $\Gamma;$

$(ii) \quad P$ is approximating as in Definition \ref{d1} for $-R$ from Lemma \ref{Ma}. That is, the coefficients of $P$ satisfy the system \eqref{A} with left hand side $-A_{\sigma l}$ where $A_{\sigma l}$ is given in Lemma \ref{Ma} (with $e=e_n$).

\

 Condition $(i)$ says that on $\Gamma$
 \begin{equation}\label{W1}
 \p _\nu W_{Q,P}=O(|X|^{k+1+\alpha}),
 \end{equation}
 and condition $(ii)$, in view of Lemma \ref{Ma}, gives 
 \begin{equation}\label{W2}
 \Delta (u_nW_{Q,P})= O(\frac{U_0}{r} |X|^{k+\alpha}).
 \end{equation}

We write below the two conditions above in terms of the coefficients of $Q$ and $P$. For convenience, we relabel the coefficients as
$$b_{\mu,0}:=q_\mu, \quad \quad b_{\mu, m+1}:=a_{\mu m} \quad \mbox{for $m \ge 0$.}$$

Precisely $(i)$ says that by taking $r=0$ and $x_n=g(x')$ in $P_*P$, then $P_*P$ vanishes of order $k+1$ (in $|x'|$) at the origin. Hence, by looking at the coefficient of $x^{\mu'}$ we find
\begin{equation}\label{b1}
b_{(\sigma', 0),1} = {\widehat c}_{\sigma'}^\mu \, \, b_{\mu, 1}, \quad \quad {\widehat c}_{\sigma'}^\mu \ne0 \quad \mbox{only if $|\mu| < |\sigma'|$},
\end{equation}
with ${\widehat c}_{\sigma'}^\mu$ depending on the derivatives of $g$ and the coefficients of $P_*$. Thus $b_{\mu,1}$ are determined uniquely from the linear system \eqref{b1} once $b_{\mu, 1}$ with $\mu_n \ne 0$ have been fixed.

Property $(ii)$ can be written as
\begin{align}\label{b} \nonumber 0  =  (l&+1)(l+2 + 2\sigma_n) \, b_{\sigma,l+2}  + (\sigma_n +1) b_{\sigma +\bar n,l+1} +\\
  &+ (\sigma_i +1)(\sigma_i+2) b_{\sigma+2\bar i,l} +  \bar c ^{\mu m}_{\sigma l} b_{\mu,m}, \quad \quad \quad \mbox{if $(\sigma_n,l) \ne (0,0)$;}
 \end{align}
\begin{align*} 0  =  (l&+1)(l+2 + 2\sigma_n) \, b_{\sigma,l+2}  + (\sigma_n +1) b_{\sigma +\bar n,l+1} +\\
 \nonumber &+ \frac 12 (\sigma_i +1)(\sigma_i+2) b_{\sigma+2\bar i,l} + \bar c ^{\mu m}_{\sigma l} b_{\mu,m},\quad \quad \quad \mbox{if $(\sigma_n,l)=(0,0)$,}
 \end{align*}
with $$\bar c ^{\mu m}_{\sigma l} \ne 0 \quad \quad \mbox{only if} \quad |\mu| +m < \sigma + l +2. $$

In this system $b_{\sigma,l+2}$ is determined by a linear combination of $b_{\mu, m}$'s with $|\mu|+m \le |\sigma| + l + 2$ and in case of equality, of $b_{\mu,m}$'s with $m < l+2$. Thus the coefficients $b_{\mu,m}$ are determined uniquely from this system once $b_{\mu,0}$ and $b_{\mu,1}$ have been fixed.

In conclusion all coefficients $b_{\mu,m}$  are determined uniquely from the two linear systems \eqref{b1}-\eqref{b} once $b_{\mu,0}$ and $b_{\mu, 1}$ with $\sigma_n \neq 0,$ are given. Notice that, by definition, we always take $b_{\mu,0}=0$ if $\mu_n \neq 0$.

Now we proceed with the proof of the theorem. After an initial dilation, we may suppose that we are close enough to the linear case, that is (see \eqref{ua}),
$$\|g\|_{C^{k+2,\alpha}(B'_1)} \le \delta, \quad \quad u_n=\frac{U_0}{r}\left(P^n_0(x,r) + \delta O(|X|^{k+1+\alpha})\right),$$
 for some polynomial $P^n_0$ of degree $k+1$ with
 $$\|P^0_n-\frac 12 \| \le \delta.$$
Using this in Lemma \ref{Ma} and in \eqref{pst}, we find that
$$|\bar c^{\mu m}_{\sigma l}| \le C \delta, \quad \quad |{\widehat c}^\mu_{\sigma'}| \le C \delta,$$
and \eqref{W1}-\eqref{W2} hold with the right hand side multiplied by $\delta$, that is
$$
 \p _\nu W_{Q,P}=\delta O(|X|^{k+1+\alpha}),
$$
$$
 \Delta (u_nW_{Q,P})= \delta O(\frac{U_0}{r} |X|^{k+\alpha}).
 $$

It suffices to show that if $w$ satisfies \eqref{NP} and
$$\left|w -W_{Q,P} \right| \le \lambda^{k+2+\alpha} \quad \mbox{in $B_\lambda$}, \quad \lambda \le 1,$$
for some approximating pair $(Q,P)$ with $\|Q\|,\|P\| \le 1$, then
$$\left|w -W_{\bar Q, \bar P} \right| \le (\rho \lambda)^{k+2+\alpha} \quad \mbox{in $B_\lambda$}, \quad \quad \|(\bar Q+r\bar P)-(Q+rP)\|_{L^\infty(B_\lambda)} \le C \lambda^{k+2+\alpha},$$
for some approximating pair $(\bar Q,\bar P)$. Then the theorem follows by applying this result inductively by starting with the initial approximating pair $(0,0)$ in $B_1$.

We prove the claim above similarly as in Lemma \ref{Imp}. We write $$w= W_{Q,P} + \lambda^{k+2+\alpha} \tilde w(X/\lambda).$$

Then $\|\tilde w\|_{L^\infty(B_1)} \le 1$ and $$\Delta (\tilde u_n \tilde w)=\lambda^{\frac 12 -k-\alpha} \Delta (u_n W_{P,Q}), \quad \quad \quad \tilde w_\nu=\lambda^{-(k+1+\alpha)} \, \p_\nu(W_{Q,P}).$$
Using \eqref{W1}-\eqref{W2} we have that in $B_1$,
$$|\Delta (\tilde u_n \tilde w)| \leq C \delta \frac{\tilde{U_0}}{r}, \quad \quad |\tilde w_\nu| \leq \delta \quad \mbox{on $\tilde \Gamma$}.$$

Thus by the compactness Lemma \ref{comp} and Theorem \ref{thin0L}
$$|\tilde w -  \tilde Q(x')-r \tilde P| \leq \frac{1}{4} \rho^{k+2+\alpha}  +C\rho^{k+3} \le \frac 12 \rho^{k+2+\alpha} \quad \text{in $B_\rho$,}$$
and $(\tilde Q,\tilde P)$ solves the system \eqref{W1}-\eqref{W2} with vanishing constants $\widehat{c}^{\mu}_{\sigma'}$ $\bar c_{\sigma l}^{\mu m}$.
As before, we can modify $(\tilde Q,\tilde P)$ above into $(\bar Q, \bar P)$ such that
$$\|(\tilde Q+r\tilde P)-(\bar Q+r\bar P)\| \le C \delta, \quad \mbox{and $(\bar Q,\bar P)(X/\lambda)$ is approximating.}$$
By taking $\delta$ sufficiently small we obtain
$$\left|\tilde w -\left(E(\bar Q) + \frac{\tilde U_0}{u_n} \bar P\right)\right|\leq \rho^{k+2} \quad \quad \mbox{in $B_\rho$,}$$
which gives the desired claim.

\qed

\section{Appendix}

We now prove our estimates for the constant coefficients case (see Section 4.3 for the statements.)

\

\noindent {\it Proof of Theorem $\ref{SchauderL}$.}
The function $u$ is uniformly H\"older continuous on compact sets of $B_1$. Moreover, since the equation is invariant after differentiating in the $x'$ direction we find
$$\|D^\mu_{x'} u\|_{C^\beta (B_{1/2})} \le C(|\mu|), \quad \quad \mu=(\mu_1,..,\mu_{n-1},0,0),$$
i.e. $u$ is $C^\infty$ in the $x'$ variable. We write the equation as
$$\Delta_{(x_n,x_{n+1})}u=-\Delta_{x'}u=:f(x),$$
and due to the invariance of the equation in the $x'$ direction, $f$ and $u$ have the same regularity properties. In particular they depend in a $C^\infty$ fashion on the $x'$ variable.

 We determine the behavior of $u$ in the $(x_n,x_{n+1})$ variables by solving the Laplace equation above in each two dimensional plane $x'=const.$
Using the complex change of variables $ z \to z^2$ i.e.
$$\bar u(z):=u(z^2), \quad \bar f(z):=f(z^2) \quad \quad z:=x_n +i x_{n+1},$$
we find
$$\Delta \bar u=4|z|^2 \bar f,$$
and $\bar u$ vanishes on $x_n=0$. After an odd reflection with respect to $x_n$, we see that the equation above is satisfied for functions $\bar u$ and $\bar f$ which are even in $x_{n+1}$, odd in $x_n$ and have the same regularity properties. This easily implies that $\bar u$ and $\bar f$ are $C^\infty$ in $z$. Moreover, $\bar u$ has a polynomial expansion at $0$ of the type
$$\bar u= x_n \left( P(x_n^2,x_{n+1}^2) +O\left(|z|^{2k+2}\right) \right), \quad \quad \deg P=k.$$
We obtain the desired result by writing $P(x_n^2,x_{n+1}^2)$ as a polynomial of degree $k$ in the variables $Re \, z^2= x_n^2-x_{n+1}^2$ and $|z|^2=x_n^2+x_{n+1}^2$, and then scaling back to $u$.

The claim that $U_0P_0$ is harmonic in $B_1 \setminus \mathcal P$ follows from scaling. Indeed let $P_0= \sum_{m=0}^k p_0^m(x,r)$ with each $p_0^m$ a homogeneous polynomial of degree $m$. We argue by induction on $m$. Clearly the statement is true for $m=0.$ Assume it is true for all $m \leq l<k.$ Then,
\begin{equation}\label{scal}v:=u - U_0\sum_{m=0}^{l} p_0^m = U_0(p_0^{l+1}(x,r) + o(|X|^{l+1}))\end{equation}
and the function $v$ is harmonic. We rescale,
$$v_\lambda(X) = \frac{v(\lambda X)}{\lambda^{1/2+l+1}}$$ and obtain a sequence of harmonic functions which by \eqref{scal} tend to $U_0 p_0^{l+1}$ as $\lambda \to 0.$ Thus $U_0 p_0^{l+1}$ is harmonic as well.
\qed

\

\noindent{\it Proof of Theorem $\ref{thin0L}$.}
This was proved in \cite{DS1}, and its proof is similar to the proof above. We sketch below a slightly different proof that uses Theorem \ref{Schauder}.

In Lemma \ref{hol} we already obtained uniform H\"older continuity of solutions $w$ on compact sets. By the invariance of the equation in the $x'$ direction we obtain that $w$ depends in a $C^\infty$ fashion in the $x'$ variable. Using barriers similar to the ones in Lemma \ref{comp} one can easily obtain that
\begin{equation}\label{wi}
|w(X)-w(x',0,0)| \le Cr \quad \quad \mbox{in $B_{1/2}$}.
\end{equation}
Then the function
$$\bar w(X):=w(X)-w(x',0,0),$$
satisfies
$$\Delta (\frac {U_0}{r} \bar w)= \frac{U_0}{r} f(x') \quad \quad f(x'):=- \Delta_{x'} w(x',0,0) \in C^\infty,$$
and, by \eqref{wi}, $v:=(U_0/r) \bar w$ vanishes continuously on $\Gamma$. We may apply Theorem \ref{Schauder} to $v$ and obtain
$$\frac{\bar w}{r}= \bar P (x,r) + O(|X|^k), \quad \quad \deg \bar P=k-1,$$
and the theorem is proved by writing
$$w(x',0,0)=\bar Q(x') + O(|x'|^{k+1}), \quad \quad \deg \bar Q=k.$$

The claim that $T$ solves the same problem as $w$ now follows from scaling, as in the final part of the previous proof.
\qed

\

We conclude this appendix, with the proof of our needed version of the Whitney Extension Theorem.

\

\noindent{\it Proof of Theorem $\ref{WE}.$}
In our case, the extension $E(Q)$ can be constructed by a convolution type operator.
Let $\rho$ be a smooth function with support in $B_{1/2} \subset \R^n$, such that
$$\int_{\R^n} \rho \, \, dx=1\quad \mbox{and} \quad \quad \int_{\R^n} \rho \,  x^\mu \, \, dx =0 \quad \mbox{if} \quad 1\leq |\mu| \leq k+2.$$

Then polynomials of degree $k+2$ are left invariant after convolution with $\rho$: $$P   *\rho = P \quad \quad \deg P=k+2.$$

Define $$E(Q) \, (x) : = \int Q(y) \, \, \rho \left(\frac{x-y}{d} \right)d^{-n} \, \, dy,$$ where $d$ denotes the distance from $x$ to $\Gamma$.

We show that $E(Q)$ satisfies the required properties. It suffices to show that for indices $\mu$ with $|\mu|=k+2$, say in $B_{\lambda/2}( \lambda e_n)$ with $\lambda$ small, we have
$$[D^\mu E(Q) ]_{C^\alpha(B_{\lambda/2}(\lambda e_n))} \le C, \quad |D^\mu E(Q)(\lambda e_n) - D^\mu Q(0)| \le C \lambda^\alpha.$$

Since $Q$ is pointwise $C^{k+2,\alpha}$ at the origin we have
$$Q=P_0+h, \quad \quad \deg P_0=k+2, \quad \|h\|_{L^\infty(B_{2\lambda})}\leq C\lambda^{k+2+\alpha},$$ and $$E(Q)=P_0+ E(h).$$

We need to show that in $B_{\lambda/2}( \lambda e_n)$,
$$[D^\mu \, E(h)]_{C^\alpha} \leq C, \quad \quad |D^\mu \, E(h)| \le C \lambda^\alpha, \quad \quad |\mu|=k+2.$$
Indeed, after a dilation of factor $1/\lambda$ we have for $x \in B_{1/2}(e_n)$
$$\overline{E(h)} \, (x) := E(h) \, (\lambda x) = \int_{B_2} h(\lambda y) \, \rho \left(\frac{x-y}{d} \right) d^{-n} \, dy,$$ with $d$ the distance from $x$ to $\Gamma/\lambda$.
Notice that $d \in C^{k+2,\alpha}_x$ hence $$\rho \left(\frac{x-y}{d}\right) d^{-n} \quad \quad \mbox{has bounded} \quad C_x^{k+2,\alpha} \quad \mbox{norm in $B_{1/2}(e_n)$}.$$ Thus, by using the bound on $h$ we find
$$\|D^\mu \, \, \overline{E(h)}\|_{C^\alpha (B_{1/2}(e_n))} \le C \lambda^{k+2+\alpha},$$
which gives the desired result.
\qed

\end{document}